\documentclass[runningheads,11pt]{llncs}
\usepackage{amssymb}
\usepackage{comment}
\usepackage{enumerate}
\usepackage{proof}
\usepackage{stmaryrd}
\usepackage[color,matrix,arrow]{xy}
\usepackage{times}
\usepackage{color}
\usepackage{wrapfig}
\usepackage{bussproofs}
\usepackage{graphicx}
\usepackage{textcomp}
\usepackage{amssymb}
\usepackage{graphicx}
\usepackage{textcomp}
\usepackage{amsmath}
\usepackage{stmaryrd}
\usepackage{url}

\usepackage{multicol}

\usepackage[left=3.4cm,right=3.4cm,top=3.4cm,bottom=3cm]{geometry}

\usepackage{setspace}
\newcommand{\bigamp}{\mathop{\mbox{\Large \&}}}

\newcommand{\mf}{\mathfrak}

\newcommand{\msf}{\mathsf}
\newcommand{\mbf}{\mathbf}
\newcommand{\mrm}{\mathrm}
\newcommand{\bu}{\bullet}
\newcommand{\mc}{\mathcal}

\newcommand{\imp}{\rightarrow}
\newcommand{\limp}{\leftarrow}

\newcommand{\Imp}{\Rightarrow}
\newcommand{\sub}{\subseteq}

\newcommand{\ol}{\overline}

\newcommand{\ua}{\uparrow}

\newtheorem{fact}[theorem]{Fact}

\title{Unified Correspondence and Proof Theory for Strict Implication}

\author{Minghui Ma\inst{1} \and Zhiguang Zhao\inst{2}}
\institute{
Institute for Logic and Intelligence, Southwest University, Chongqing, China\\
\email{mmh.thu@gmail.com}
\and
Delft University of Technology, Delft, The Netherlands\\
\email{zhaozhiguang23@gmail.com}
}

\authorrunning{Ma and Zhao}

\begin{document}
\maketitle
\begin{abstract}
The unified correspondence theory for distributive lattice expansion logics (DLE-logics) 
is specialized to strict implication logics. As a consequence of a general semantic consevativity result, a wide range of strict implication logics can be conservatively extended to Lambek Calculi over the bounded distributive full non-associative Lambek calculus ($\msf{BDFNL}$). Many strict implication sequents can be transformed into analytic rules employing one of the main tools of unified correspondence theory, namely (a suitably modified version of) the Ackermann lemma based algorithm $\msf{ALBA}$. Gentzen-style cut-free sequent calculi for $\msf{BDFNL}$ and its extensions with analytic rules which are transformed from strict implication sequents, are developed.
\end{abstract}

\section{Introduction}\label{section:introduction}
Strict implication is an intensional implication which is semantically interpreted on Kripke binary relational models in the same fashion as intuitionistic implication. Kripke frames for intuitionistic logic are partially ordered sets, and valuations are required to be persistent, i.e., to map propositional variables to upsets. The intuitionistic implication is already an example of strict implication. 
Subintuitionistic logics, which are
prime examples of strict implication logics (cf.~\cite{corsi87,Visser81,AR98,CJ03,Bou04,Res94,Fran07,Fab04,AB62,Kik01}), arise semantically 
by dropping some conditions from the intuitionistic models outlined above, such as the requirement that the accessibility relation to be reflexive or transitive, and the persistency of valuations.
For example, 
Visser's basic propositional logic $\msf{BPL}$ \cite{Visser81} is a subintuitionistic logic characterized by the class of all transitive frames under the semantics by dropping only the reflexivity condition on frames from the intuitionistic case, and it is embedded into the normal modal logic $\msf{K4}$ via the G\"{o}del-McKinsey-Tarski translation. Another example is the least subintuitionistic logic $\msf{F}$ introduced by Corsi \cite{corsi87} which is characterized by the class of all Kripke frames under the semantics by dropping all conditions on frames or models. Naturally, $\msf{F}$ is embeddable into the least normal modal logic $\msf{K}$.

The present paper proposes a uniform approach to the proof theory of the family of strict implication logics. 
Cut-free sequent calculi exist in the literature for some members of this family \cite{IK07},
for instance, for Visser's propositional logics
\cite{IKK01}. These calculi lack a left- and a right-introduction rule for $\imp$. Instead, 
there is only one rule in which $2^n$ premisses are needed when the conclusion has $n$ implication formulas as the antecedent of the sequent. 
In contrast with this, in the present paper, we provide modular cut-free calculi for a wide class of strict implication logics, each of which has the standard left- and right-introduction rules. Our methodology uses unified correspondence theory. It takes the move from some general semantic conservativity results which naturally arise from the semantic environment of unified correspondence. Specifically, we use the fact that certain strict implication logics can be conservatively extended to 
suitable axiomatic extensions of the bounded distributive lattice full non-associative Lambek calculus ($\msf{BDFNL}$)\footnote{~Non-associative Lambek calculus was first developed by Lambek \cite{Lambek58,Lambek61}. For details about Lambek calculi and substructural logics, we refer to \cite{GO07,Bus06,BusF09,Morrill2010}.} , and develop Gentzen-style cut-free sequent calculi for these axiomatic extensions, using the tools of unified correspondence.

Let us first explain what unified correspondence is and how it can be used in proof theory. In recent years, based on duality-theoretic insights \cite{ConPalSou12}, an encompassing perspective has emerged, making it possible to export the Sahlqvist theory from modal logic to a wide range of logics which includes, among others, intuitionistic and distributive lattice-based (normal modal) logics \cite{CoPa12}, non-normal (regular) modal logics \cite{PaSoZh15r}, substructural logics \cite{CoPa11}, hybrid logics \cite{ConRob}, and mu-calculus \cite{CoCr14,CFPS}. This work has stimulated many applications. Some are closely related to the core concerns of the theory itself, such as the understanding of the relationship between different methodologies for obtaining canonicity results \cite{PaSoZh14}, or of the pseudo-correspondence \cite{CGPSZ14}. Other applications include the dual characterizations of classes of finite lattices \cite{FrPaSa14}, computing the first-order correspondence of rules for one-step frames \cite{BeGh14,La15}, and the identification of the syntactic shape of axioms which can be translated into analytic structural rules\footnote{~Informally, \emph{analytic} rules are those which can be added to a display calculus with cut elimination obtaining again a display calculus with cut elimination.} of a proper display calculus \cite{GP2015}.
These results have given rise to the theory called \emph{unified correspondence} \cite{CoGhPa13}.
	
The most important technical tools in unified correspondence are: (a) a very general syntactic definition of Sahlqvist formulas, which applies uniformly to each logical signature and is given purely in terms of the order-theoretic properties of the algebraic interpretations of the logical connectives; (b) the Ackermann lemma based algorithm $\msf{ALBA}$, which effectively computes first-order correspondents of input term-inequalities, and is guaranteed to succeed on a wide class of inequalities (the so-called \emph{inductive} inequalities) which, like the Sahlqvist class, can be defined uniformly in each signature, and which properly and significantly extends the Sahlqvist class.

From the point of view of unified correspondence, the family of strict implication logics is a very interesting
subclass of normal DLE-logics (i.e., logics algebraically identified by varieties of bounded distributive lattice expansions), not only because they are very well-known and very intensely investigated, but also because they are enjoying two different and equally natural relational semantics, namely, the one described above, interpreting the binary implication by means of a {\em binary} relation \cite{CJ03}, and another, arising from the standard treatment of binary modal operators, interpreting the binary implication by means of a {\em ternary} relation \cite{Kur94}. The existence of these two different semantics makes unified correspondence a very appropriate tool to study the Sahlqvist-type theory of these logics, because of one of the features specific to unified correspondence theory, namely the possibility of developing Sahlqvist-type theory for the logics of strict implication in a modular and simultaneous way for their two types of relational semantics. 

In the present paper we specialize the two tools of unified correspondence theory from the general setting of normal DLE-logics to the setting of strict implication logics. 
The semantic environment of unified correspondence theory allows for a general semantic conservativity result for normal DLE logics, which has been briefly outlined in \cite{GP2015} and is further clarified in the present paper (cf.\ Theorem \ref{th:conservative extension}), and specialized to the setting of strict implication logics.

A second reason for exploring strict implication logics with the tools of unified correspondence is given by the recent developments mentioned above, establishing systematic connections between correspondence results for normal DLE-logics and the characterization of the axiomatic extensions of basic normal DLE-logics which admit display calculi with cut elimination. In particular, in \cite{GP2015}, the tool (a) of unified correspondence theory has been used to provide the syntactic characterization of those axioms which correspond to analytic rules, and tool (b) has been used to provide an effective computation of the rules corresponding to each analytic axiom. 
This work provides an exhaustive answer, relative to the setting of display calculi, to a key question in structural proof theory
which has been intensely investigated in various proof-theoretic settings
(cf.\ \cite{negri2005proof,ciabattoni2008axioms,ciabattoni2009expanding,gore2011correspondence,ciabattoni2012algebraic,lellmann2013correspondence,lahav2013frame,marin2014label,lellmann2014axioms}). 

In fact, a major conceptual motivation of the present paper is provided by the insight that the unified correspondence methodology can be applied to the analyticity issue also in proof-theoretic settings different from display calculi.
Following this insight, in the present paper, we use the tools of unified correspondence in two different ways. Firstly, we present a modified version of the algorithm $\msf{ALBA}$ which is specific to the task of the direct computation of analytic rules of a Gentzen-style calculus for certain logics of strict implication. Secondly, we use this algorithm as a calculus not only to compute analytic rules, but also to establish semantic (algebraic), hence logical equivalences between axioms of different but related logical signatures. This latter one is a novel application of unified correspondence. 

\textbf{Structure of the paper.}  In section 2, we will summarize unified correspondence theory for DLE-logics with specialization to strict implication logics. Specifically, a general theorem on semantic conservativity, $\msf{ALBA}$ algorithm and first-order correspondence will be formulated and specialized.
In section 3, we will introduce the Ackermann lemma based calculus $\msf{ALC}$ for calculating correspondence on over algebras between the strict implication language $\mc{L}_\mrm{SI}$ and the language $\mc{L}_\bu$. More conservativity results will be obtained by using $\msf{ALC}$. In section 4, we will develop cut-free Gentzen-style sequent calculus for $\msf{BDFNL}$, and then extend it
with analytic rules to obtain cut-free sequent calculi.

\section{Preliminaries}
In this section, we will summarize the unified correspondence theory for normal DLE-logics from \cite{GP2015} with specialization to strict implication logics. 

\subsection{Syntax and semantics for DLE-logics}
An {\em order-type} over $n\in \mathbb{N}$
is an $n$-tuple $\varepsilon\in \{1, \partial\}^n$.
Order-types will be typically associated with arrays of variables $\vec p: = (p_1,\ldots, p_n)$. When the order of the variables in $\vec p$ is not specified, we will sometimes abuse notation and write $\varepsilon(p) = 1$ or $\varepsilon(p) = \partial$.
For every order type $\varepsilon$, we denote its {\em opposite} order type by $\varepsilon^\partial$, that is, $\varepsilon^\partial_i = 1$ iff $\varepsilon_i=\partial$ for every $1 \leq i \leq n$. For any lattice $A$, we let $A^1: = A$ and $A^\partial$ be the dual lattice, that is, the lattice associated with the converse partial order of $A$. For any order type $\varepsilon$, we let $A^\varepsilon: = \Pi_{i = 1}^n A^{\varepsilon_i}$.

The language $\mathcal{L}_\mathrm{DLE}(\mathcal{F}, \mathcal{G})$ (sometimes abbreviated as $\mathcal{L}_\mathrm{DLE}$) consists of: 1) a denumerable set of proposition letters $\mathsf{AtProp}$, elements of which are denoted $p,q,r$, possibly with indexes; 2) disjoint finite sets of connectives $\mathcal{F}$ and $\mathcal{G}$. Each $f\in \mathcal{F}$ (respectively\ $g\in \mathcal{G}$) has arity $n_f\in \mathbb{N}$ (respectively\ $n_g\in \mathbb{N}$) and is associated with some order-type $\varepsilon_f$ over $n_f$ (respectively\ $\varepsilon_g$ over $n_g$). 
\begin{definition}
The {\em terms} ({\em formulas}) of $\mathcal{L}_\mathrm{DLE}$ are defined recursively as follows:
\[
\phi ::= p \mid \bot \mid \top \mid (\phi \wedge \phi) \mid (\phi \vee \phi) \mid f(\overline{\phi}) \mid g(\overline{\phi})
\]
where $p \in \mathsf{AtProp}$, $f \in \mathcal{F}$, $g \in \mathcal{G}$. Terms (formulas) in $\mathcal{L}_\mathrm{DLE}$ will be denoted either by $s,t$, or by lowercase Greek letters such as $\varphi, \psi, \gamma$ etc. An $\mc{L}_\mrm{DLE}$-sequent is an expression of the form $\phi\vdash\psi$.
\end{definition}

\begin{definition}
For any tuple $(\mathcal{F}, \mathcal{G})$ of disjoint sets of function symbols as above, a {\em distributive lattice expansion} (abbreviated as DLE) is a tuple $(A, \mathcal{F}^A, \mathcal{G}^A)$ such that $A$ is a bounded distributive lattice, $\mathcal{F}^A = \{f^A\mid f\in \mathcal{F}\}$ and $\mathcal{G}^A = \{g^A\mid g\in \mathcal{G}\}$, such that every $f^A\in\mathcal{F}^A$ (respectively\ $g^A\in\mathcal{G}^A$) is an $n_f$-ary (respectively\ $n_g$-ary) operation on $A$. A DLE $(A, \mathcal{F}^A, \mathcal{G}^A)$ is {\em normal} if every $f^A\in\mathcal{F}^A$ (respectively\ $g^A\in\mathcal{G}^A$) preserves finite joins (respectively\ meets) in each coordinate with $\varepsilon_f(i)=1$ (respectively\ $\varepsilon_g(j)=1$) and reverses finite meets (respectively\ joins) in each coordinate with $\varepsilon_f(i)=\partial$ (respectively\ $\varepsilon_g(j)=\partial$).
\end{definition}

For each operator $f\in\mc{F}$ (respectively $g\in\mc{G}$) and $1\leq i\leq n_f$ (respectively $1\leq j\leq n_g$), let $f_i[-]$ (respectively $g_j[-]$) be the operator $f$ (respectively $g$) with a hole at the $i$-coordinate (respectively the $j$-th coordinate), and other coordinates be parameters. Let $f_i[a]$ (reap. $g_j[a]$) be the value of $f$ (respectively $g$) when the hole is given the input $a$. The class of all normal DLEs, denoted by $\mathbb{DLE}$, is equationally definable by distributive lattice identities and the following equations for any $f\in \mathcal{F}$ (respectively\ $g\in \mathcal{G}$) and $1\leq i\leq n_f$ (respectively\ $1\leq j\leq n_g$):
\begin{itemize}
\item[(1)] if $\varepsilon_f(i) = 1$, then $f_i[a\vee b] = f_i[a]\vee f_i[b]$ and $f_i[\bot] = \bot$,
\item[(2)] if $\varepsilon_f(i) = \partial$, then $f_i[a\wedge b] = f_i[a]\vee f_i[b]$ and $f_i[\top] = \bot$,
\item[(3)] if $\varepsilon_g(j) = 1$, then $g_j[a\wedge b] = g_j[a]\wedge g_j[b]$ and $g_j[\top] = \top$,
\item[(4)] if $\varepsilon_g(j) = \partial$, then $g_j[a\vee b] = g_j[a]\wedge g_j[b]$ and $g_j[\bot] = \top$.
\end{itemize}
Each language $\mathcal{L}_\mathrm{DLE}$ is interpreted in the appropriate class of normal DLEs. In particular, for every DLE $A$, each operation $f^A\in \mathcal{F}^A$ (respectively\ $g^A\in \mathcal{G}^A$) is finitely join-preserving (respectively\ meet-preserving) in each coordinate when regarded as a map $f^A: A^{\varepsilon_f}\to A$ (respectively\ $g^A: A^{\varepsilon_g}\to A$).
	
\begin{definition}
\label{def:DLE:logic:general}
For any language $\mathcal{L}_\mathrm{DLE}(\mathcal{F}, \mathcal{G})$, the minimal DLE-logic is the set of $\mathcal{L}_\mathrm{DLE}$-sequents $\phi\vdash\psi$, which contains the following axioms:

(1) Sequents for lattice connectives:
\begin{align*}
&p\vdash p, && \bot\vdash p, && p\vdash \top, & & p\wedge (q\vee r)\vdash (p\wedge q)\vee (p\wedge r), &\\
&p\vdash p\vee q, && q\vdash p\vee q, && p\wedge q\vdash p, && p\wedge q\vdash q, &
\end{align*}

(2) Sequents for connectives $f\in\mc{F}$ and $g\in \mc{G}$:
\[
\begin{tabular}{|c|c|}
\hline
$\varepsilon_f(i) = 1$ & $\varepsilon_f(i) = \partial$\\
\hline
$f_i[\bot] \vdash \bot$&
$f_i[\top] \vdash \bot$\\
$f_i[p\vee q] \vdash f_i[p]\vee f_i[q]$&
$f_i[p\wedge q] \vdash f_i[p]\vee f_i[q]$\\
\hline
$\varepsilon_g(j) = 1$ & $\varepsilon_g(j) = \partial$\\
\hline
$\top\vdash g_j[\top]$& 
$\top\vdash g_j[\bot]$\\
$g_j[p]\wedge g_j[q]\vdash g_j[p\wedge q]$& 
$g_j[p]\wedge g_j[q]\vdash g_j[p\vee q]$\\
\hline
\end{tabular}
\]
and is closed under the following inference rules:
\begin{displaymath}
\frac{\phi\vdash \chi\quad \chi\vdash \psi}{\phi\vdash \psi}
\quad
\frac{\phi\vdash \psi}{\phi[\chi/p]\vdash\psi[\chi/p]}
\quad
\frac{\chi\vdash\phi\quad \chi\vdash\psi}{\chi\vdash \phi\wedge\psi}
\quad
\frac{\phi\vdash\chi\quad \psi\vdash\chi}{\phi\vee\psi\vdash\chi}
\end{displaymath}
\begin{displaymath}
\frac{\phi\vdash\psi}{f_i[\phi]\vdash f_i[\psi]}{~(\varepsilon_f(i) = 1)}
\quad
\frac{\phi\vdash\psi}{f_i[\psi]\vdash f_i[\phi]}{~(\varepsilon_f(i) = \partial)}
\end{displaymath}
\begin{displaymath}
\frac{\phi\vdash\psi}{g_j[\phi]\vdash g_j[\psi]}{~(\varepsilon_g(j) = 1)}
\quad
\frac{\phi\vdash\psi}{g_j[\psi]\vdash g_j[\phi]}{~(\varepsilon_g(j) = \partial)}.
\end{displaymath}
The formula $\phi[\chi/p]$ is obtained from $\phi$ by substituting $\chi$ for $p$ uniformly. The minimal DLE-logic is denoted by $\mathbf{L}_\mathbb{DLE}$. For any DLE-language $\mathcal{L}_{\mrm{DLE}}$, by a {\em $\mathrm{DLE}$-logic} we understand any axiomatic extension of $\mathbf{L}_\mathbb{DLE}$.
\end{definition}

A sequent $\phi\vdash\psi$ is valid in a DLE $(A, \mathcal{F}^A, \mathcal{G}^A)$ if $\mu(\phi)\leq \mu(\psi)$ for every homomorphism $\mu$ from the $\mathcal{L}_\mathrm{DLE}$-algebra of formulas over $\mathsf{AtProp}$ to $A$. The notation $\mathbb{DLE}\models\phi\vdash\psi$ indicates that $\phi\vdash\psi$ is valid in every DLE. Then, by means of a routine Lindenbaum-Tarski construction, it is easy to show that the minimal DLE-logic $\mathbf{L}_\mathbb{DLE}$ is sound and complete with respect to its corresponding class of $\mc{L}_\mrm{DLE}$-algebras $\mathbb{DLE}$, i.e.\ that any sequent $\phi\vdash\psi$ is provable in $\mathbf{L}_\mathbb{DLE}$ if and only if $\mathbb{DLE}\models\phi\vdash\psi$.  

We will now specialize normal DLE-logics to strict implication logics. The {\em strict implication language} $\mc{L}_\mrm{SI}$ is identified with the DLE-language $\mathcal{L}_\mathrm{DLE}(\mathcal{F}, \mathcal{G})$ where $\mc{F}=\emptyset$ and $\mc{G} = \{\imp\}$.
The order-type of $\imp$ is $(\partial, 1)$.
The definition of normal DLE-algebra is specialized into the following definition:

\begin{definition}
An algebra $\mf{A} = (A, \wedge, \vee, \bot, \top, \imp)$ is called a {\em bounded distributive lattice with strict implication} (BDI) if its $(\wedge,\vee,\bot,\top)$-reduct is a bounded distributive lattice and $\imp$ is a binary operation on $A$ satisfying the following conditions for all $a,b,c\in A$:
\begin{itemize}
\item[](C1) $(a\imp b)\wedge(a\imp c) = a\imp (b\wedge c)$,
\item[](C2) $(a\imp c)\wedge(b\imp c) = (a\vee b)\imp c$,
\item[](C3) $a\imp \top = \top = \bot\imp a$.
\end{itemize}
Let $\mathbb{BDI}$ be the class of all BDIs. Henceforth, we also write a BDI as $(A, \imp)$ where $A$ is supposed to be a bounded distributive lattice.
\end{definition}

\begin{definition}
The algebraic sequent system $\mathbf{S}_\mathbb{BDI}$ consists of the following axioms and rules:
\begin{itemize}
\item Axioms:
\begin{align*}
&\mrm{(Id)}~\phi\vdash\phi,
\quad 
\mrm{(D)}~\phi\wedge(\psi\vee\gamma)\vdash(\phi\wedge\psi)\vee(\phi\wedge\gamma),
\\
&
(\top)~\phi\vdash\top,
\quad
(\bot)~\bot\vdash\phi,
\quad
\mrm{(N_\top)}~\top\vdash \phi\imp\top,
\quad
\mrm{(N_\bot)}~\top\vdash \bot\imp\phi,
\\
&
(\mrm{M_1})~(\phi\imp\psi)\wedge(\phi\imp\gamma)\vdash\phi\imp(\psi\wedge\gamma),
\\
&
(\mrm{M_2})~(\phi\imp\gamma)\wedge(\psi\imp\gamma)\vdash(\phi\vee\psi)\imp\gamma,
\end{align*}
\item Rules:
\[
(\mrm{M_3})~\frac{\phi \vdash\psi}{\chi\imp\phi\vdash\chi\imp\psi},
\quad
(\mrm{M_4})~\frac{\phi\vdash\psi}{\psi\imp\chi\vdash\phi\imp\chi},
\]
\[
(\wedge \mrm{L})~\frac{\phi_i\vdash\psi}{\phi_1\wedge\phi_2\vdash\psi}{(i=1,2)},
\quad
(\wedge \mrm{R})~\frac{\gamma\vdash\phi\quad\gamma\vdash\psi}{\gamma\vdash\phi\wedge\psi},
\]
\[
(\vee \mrm{L})~\frac{\phi\vdash\chi\quad\psi\vdash\chi}{\phi\vee\psi\vdash\gamma},
\quad
(\vee \mrm{R})~\frac{\psi\vdash\phi_i}{\psi\vdash\phi_1\vee\phi_2}{(i=1,2)},
\]
\[
\mrm{(cut)}~\frac{\phi \vdash\psi\quad\psi\vdash\gamma}{\phi\vdash\gamma},
\]
\end{itemize}
\end{definition}

It is easy to see that $\mathbf{S}_\mathbb{BDI}$ is a specialization of $\mbf{L}_\mathbb{DLE}$.
Some extensions of $\mathbf{S}_\mathbb{BDI}$, strict implication logics extending it, can be obtained by adding `characteristic' sequents. Table \ref{table:cs} list some characteristic sequents that are considered in literature.\footnote{These characteristic sequents may have different names or forms in literature. For example, (MP) is written as $p, p\imp q\vdash q$ where the comma means conjunction. The sequent (Fr) is named by the Frege axiom $(p\imp (q\imp r))\imp ((p\imp q)\imp (p\imp r))$.}
\begin{table}[htbp]
\caption{Some Characteristic Sequents}
\vspace{-1em}
\[
\begin{tabular}{|l|l|l|}
\hline
Name ~&~ Sequent ~&~ Literature\\
\hline
($\mrm{I}$) ~&~ $q\vdash p \imp p$~&~ \cite{CJ03}\\
($\mrm{Tr}$) ~&~ $(p\imp q)\wedge(q\imp r)\vdash p\imp r$~&~\cite{CJ03}\cite[p.44]{Res94b}\\
($\mrm{MP}$) ~&~ $p\wedge(p\imp q)\vdash q$~&~\cite{CJ03,IK07}\\
($\mrm{W}$) ~&~ $p \vdash q\imp p$~&~\cite{CJ03}\cite[p.34]{Res94b}\\
($\mrm{RT}$) ~&~ $p \imp q \vdash r\imp(p\imp q)$~&~\cite{CJ03,IK07}\\
($\mrm{B}$) ~&~ $p \imp q \vdash(r\imp p)\imp(r\imp q)$~&~\cite[p.32]{Res94b}\\
($\mrm{B}'$) ~&~ $p \imp q \vdash(q\imp r)\imp(p\imp r)$~&~\cite[p.32]{Res94b}\\
($\mrm{C}$) ~&~ $p \imp(q\imp r) \vdash q\imp(p\imp r)$~&~\cite[p.32]{Res94b}\\
($\mrm{Fr}$) ~&~ $p \imp(q\imp p) \vdash(p\imp q)\imp(p\imp r)$~&~\cite[p.44]{Res94b}\\
($\mrm{W}'$) ~&~ $p \imp(p\imp q) \vdash p\imp q$~&~\cite[p.44]{Res94b}\\
($\mrm{Sym}$) ~&~ $p\vdash((p\imp q)\imp r)\vee q$~&~\cite{IK07}\\
($\mrm{Euc}$) ~&~ $\top\vdash((p\imp q)\imp r)\vee(p\imp q)$~&~\cite{IK07}\\
($\mrm{D}$) ~&~ $\top\imp\bot\vdash\bot$~&~\cite{IK07}\\
\hline
\end{tabular}
\]
\label{table:cs}
\end{table}
For any sequent system $\mbf{S}$ and a set of sequents $\Sigma$, the notation $\mathbf{S}+\Sigma$ stands for the system obtained from $\mathbf{S}$ by adding all instances of sequents in $\Sigma$ as new axioms. Strict implication logics in Table \ref{table:sl} can be obtained using these characteristic sequents. Some of them are considered in literature.\footnote{These logics are presented in various ways in literature as Hilbert-style systems, natural deduction systems or sequent systems. The name $\msf{GK}^I$ \cite{IK07} stands for the Gentzen-style sequent calculus for the minimal strict implication logic under binary relational semantics which can be embedded into the minimal normal modal logic $\msf{K}$.} 

\begin{table}[htbp]
\caption{Some Strict Implication Logics}
\vspace{-1em}
\[
\begin{tabular}{|l|l|l|}
\hline
Name ~&~ System ~&~ Literature\\
\hline
$\mathbf{S}_\mathbb{WH}$, $\msf{GK}^I$ ~&~ $\mathbf{S}_\mathbb{BDI}+(\mrm{I})+(\mrm{Tr})$~&~ \cite{CJ03,IK07,corsi87,Dosen93,ys16}\\
$\mathbf{S}_\mathbb{T}$~&~ $\mathbf{S}_\mathbb{BDI}+(\mrm{MP})$~&~\\
$\mathbf{S}_\mathbb{W}$ ~&~ $\mathbf{S}_\mathbb{BDI}+(\mrm{W})$~&~\\
$\mathbf{S}_\mathbb{RT}$ ~&~ $\mathbf{S}_\mathbb{BDI}+(\mrm{RT})$~&~\\
$\mathbf{S}_\mathbb{B}$ ~&~ $\mathbf{S}_\mathbb{BDI}+(\mrm{B})$~&~\\
$\mathbf{S}_\mathbb{B'}$ ~&~ $\mathbf{S}_\mathbb{BDI}+(\mrm{B'})$~&~\\
$\mathbf{S}_\mathbb{C}$ ~&~ $\mathbf{S}_\mathbb{BDI}+(\mrm{C})$~&~\\
$\mathbf{S}_\mathbb{FR}$ ~&~ $\mathbf{S}_\mathbb{BDI}+(\mrm{Fr})$~&~\\
$\mathbf{S}_\mathbb{W'}$ ~&~ $\mathbf{S}_\mathbb{BDI}+(\mrm{W'})$~&~\\
$\mathbf{S}_\mathbb{SYM}$ ~&~ $\mathbf{S}_\mathbb{BDI}+(\mrm{Sym})$~&~\\
$\mathbf{S}_\mathbb{EUC}$ ~&~ $\mathbf{S}_\mathbb{BDI}+(\mrm{Euc})$~&~\\
$\mathbf{S}_\mathbb{BCA}$~&~ $\mathbf{S}_\mathbb{T}+(\mrm{W})$~&~\cite{Suzuki99,CJ03,IK07,Visser81,AR95,AR98,IKK01}\\
$\mathsf{GKT}^I$ ~&~ $\mathsf{GK}^I+(\mrm{MP})$~&~\cite{corsi87,IK07}\\
$\mathsf{GK4}^I$ ~&~ $\mathsf{GK}^I+(\mrm{RT})$~&~\cite{corsi87,IK07}\\
$\mathsf{GS4}^I$ ~&~ $\mathsf{GKT}^I+(\mrm{RT})$~&~\cite{corsi87,IK07}\\
$\mathsf{GKB}^I$ ~&~ $\mathsf{GK}^I+(\mrm{Sym})$~&~\cite{corsi87,IK07}\\
$\mathsf{GK5}^I$ ~&~ $\mathsf{GK}^I+(\mrm{Euc})$~&~\cite{IK07}\\
$\mathsf{GK45}^I$ ~&~ $\mathsf{GK5}^I+(\mrm{RT})$~&~\cite{IK07}\\
$\mathsf{GKS5}^I$ ~&~ $\mathsf{GK45}^I+(\mrm{W})$~&~\cite{IK07}\\
$\mathsf{GK4}^{I+}$ ~&~ $\mathsf{GK}^I+(\mrm{W})$~&~\cite{IK07}\\
$\mathsf{GKD}^I$ ~&~ $\mathsf{GK}^I+(\mrm{D})$~&~\cite{corsi87,IK07}\\
\hline
\end{tabular}
\]
\label{table:sl}
\end{table}

Each sequent $\phi\vdash\psi$ defines a class of BDIs. Each strict implication logic $\mbf{S}_\mathbb{BDI}+\Sigma$ defines a class of BDIs denoted by $\msf{Alg}(\Sigma)$. For example, some subvarieties 
are considered in \cite{CJ05}. A BDI $(A, \imp)$ is called a {\em weak Heyting algebra} (WH-algebra) if the following conditions are satisfied for all $a, b, c \in A$:
\begin{itemize}
\item[](C4) $b\leq a\imp a$.
\item[](C5) $(a\imp b)\wedge(b\imp c)\leq (a\imp c)$.
\end{itemize}
Let $\mathbb{WH}$ be the class of all WH-algebras. 
A $\msf{wKT}_\sigma$-algebra is a WH-algebra $(A, \imp)$ satisfying the condition $a\wedge(a\imp b)\leq b$ for all $a, b\in A$.
A {\em basic algebra} is a WH-algebra $(A,\imp)$ satisfying the condition $a\leq b\imp a$ for all $a,b\in A$. Let $\mathbb{T}$ and $\mathbb{BCA}$ be  the classes of all $\msf{wKT}_\sigma$-algebras and basic algebras respectively. The variety of Heyting algebras is a subvariety of $\mathbb{BCA}$, i.e., it is the class of all basic algebras $(A,\imp)$ satisfying the condition $\top\imp a\leq a$ for all $a\in A$ (cf.~e.g. \cite{AR98,AA04}).

As a corollary of the soundness and completeness of DLE-logics with respect to their $\mc{L}_\mrm{DLE}$-algebras, one gets the following theorem immediately:

\begin{theorem}
For any strict implication logic $\mbf{S}_\mathbb{BDI}+\Sigma$, an $\mc{L}_\mrm{SI}$-sequent $\phi\vdash\psi$ is derivable in $\mbf{S}_\mathbb{BDI}+\Sigma$ if and only if $\msf{Alg}(\Sigma)\models \phi\vdash\psi$.
\end{theorem}

\subsection{The expanded language $\mathcal{L}_\mathrm{DLE}^*$}
Any given language $\mathcal{L}_\mathrm{DLE} = \mathcal{L}_\mathrm{DLE}(\mathcal{F}, \mathcal{G})$ can be extended to the language $\mathcal{L}_\mathrm{DLE}^* = \mathcal{L}_\mathrm{DLE}(\mathcal{F}^*, \mathcal{G}^*)$, where $\mathcal{F}^*\supseteq \mathcal{F}$ and $\mathcal{G}^*\supseteq \mathcal{G}$ are obtained by expanding $\mathcal{L}_\mathrm{DLE}$ with the following connectives:

\begin{enumerate}
\item the Heyting implications $\leftarrow_H$ and $\rightarrow_H$, the intended interpretations of which are the right residuals of $\wedge$ in the first and second coordinate respectively, and $>\hspace{-0.5em}-$ and $ -\hspace{-0.5em}<$, the intended interpretations of which are the left residuals of $\vee$ in the first and second coordinate, respectively;

\item the $n_f$-ary connective $f^\sharp_i$ for $0\leq i\leq n_f$, the intended interpretation of which is the right residual of $f\in\mathcal{F}$ in its $i$th coordinate if $\varepsilon_f(i) = 1$ (respectively\ its Galois-adjoint if $\varepsilon_f(i) = \partial$);
		\item the $n_g$-ary connective $g^\flat_i$ for $0\leq i\leq n_g$, the intended interpretation of which is the left residual of $g\in\mathcal{G}$ in its $i$th coordinate if $\varepsilon_g(i) = 1$ (respectively\ its Galois-adjoint if $\varepsilon_g(i) = \partial$).
\end{enumerate}

We stipulate that $>\hspace{-0.5em}-$, $-\hspace{-0.5em}<$ $\in \mathcal{F}^*$, that $\rightarrow_H, \leftarrow_H\in \mathcal{G}^*$, and moreover, that
	$f^\sharp_i\in\mathcal{G}^*$ if $\varepsilon_f(i) = 1$, and $f^\sharp_i\in\mathcal{F}^*$ if $\varepsilon_f(i) = \partial$. Dually, $g^\flat_j\in\mathcal{F}^*$ if $\varepsilon_g(i) = 1$, and $g^\flat_j\in\mathcal{G}^*$ if $\varepsilon_g(j) = \partial$. The order-type assigned to the additional connectives is predicated on the order-type of their intended interpretations.  

\begin{definition}
For any language $\mathcal{L}_\mathrm{DLE}(\mathcal{F}, \mathcal{G})$, the {\em minimal} $\mathcal{L}_\mathrm{DLE}^*$-{\em logic} is defined by specializing Definition \ref{def:DLE:logic:general} to the language $\mathcal{L}_\mathrm{DLE}^* = \mathcal{L}_\mathrm{DLE}(\mathcal{F}^*, \mathcal{G}^*)$ and closing under the following additional rules:
\begin{enumerate}
\item Residuation rules for lattice connectives:
$$
\begin{array}{cccc}
\AxiomC{$\phi\wedge\psi\vdash \chi$}
\doubleLine
\UnaryInfC{$\psi\vdash \phi\rightarrow_H \chi$}
\DisplayProof
~
&~
\AxiomC{$\phi\wedge\psi\vdash \chi$}
\doubleLine
\UnaryInfC{$\phi\vdash \chi\leftarrow_H \psi$}
\DisplayProof
~
&~
\AxiomC{$\phi\vdash \psi\vee\chi$}
\doubleLine
\UnaryInfC{$\psi -\hspace{-0.5em}< \phi \vdash \chi$}
\DisplayProof
~&~
\AxiomC{$\phi\vdash \psi\vee\chi$}
\doubleLine
\UnaryInfC{$\phi >\hspace{-0.5em}- \chi \vdash \psi$}
\DisplayProof
\end{array}
$$
Notice that the rules for $\rightarrow_H$ and $\leftarrow_H$ are interderivable, since $\wedge$ is commutative; similarly, the rules for $-\hspace{-0.5em}<$ and $>\hspace{-0.5em}-$ are interderivable, since $\vee$ is commutative.
\item Residuation rules for $f\in \mathcal{F}$ and $g\in \mathcal{G}$:
\[
\AxiomC{$f_i[\phi] \vdash \psi$}
\doubleLine
\RightLabel{\small$(\varepsilon_f(i) = 1),$}
\UnaryInfC{$\phi\vdash f^\sharp_i[\psi]$}
\DisplayProof
\quad
\AxiomC{$\phi \vdash g_j[\psi]$}
\doubleLine
\RightLabel{\small$(\varepsilon_g(j) = 1),$}
\UnaryInfC{$g^\flat_j[\phi]\vdash \psi$}
\DisplayProof
\]
\[
\AxiomC{$f_i[\phi]\vdash \psi$}
\doubleLine
\RightLabel{\small$(\varepsilon_f(i) = \partial),$}
\UnaryInfC{$f^\sharp_i[\psi]\vdash \phi$}
\DisplayProof
\quad
\AxiomC{$\phi \vdash g_j[\psi]$}
\doubleLine
\RightLabel{\small ($\varepsilon_g(j) = \partial).$}
\UnaryInfC{$\psi\vdash g^\flat_j[\phi]$}
\DisplayProof
\]
\end{enumerate}
The double line in each rule above indicates that the rule is invertible. Let $\mathbf{L}_\mathbb{DLE}^*$ be the minimal $\mathcal{L}_\mathrm{DLE}^*$-logic. 
\end{definition}
	
The algebraic semantics of $\mathbf{L}_\mathbb{DLE}^*$ is given by the class of all $\mathcal{L}_\mathrm{DLE}^*$-algebras, defined as $(H, \mathcal{F}^*, \mathcal{G}^*)$ such that $H$ is a bi-Heyting algebra (because there are right adjoints or residuals of $\wedge$ and $\vee$ in the algebra) and moreover,
\begin{enumerate}			
\item for every $f\in \mathcal{F}$ s.t.\ $n_f\geq 1$, all $a_i, b\in H$ with $1\leq i\leq n_f$,
\begin{itemize}
\item[--] if $\varepsilon_f(i) = 1$, then $f_i[a_i]\leq b$ iff $a_i\leq f^\sharp_i[b]$;
\item[--] if $\varepsilon_f(i) = \partial$, then $f_i[a_i]\leq b$ iff $a_i\leq^\partial f^\sharp_i[b]$.
\end{itemize}
\item for every $g\in \mathcal{G}$ s.t.\ $n_g\geq 1$, any $a_j,b\in H$ with $1\leq j\leq n_g$,
\begin{itemize}
\item[--] if $\varepsilon_g(j) = 1$, then $b\leq g_j[a_j]$ iff $g^\flat_j[b]\leq a_j$.
\item[--] if $\varepsilon_g(j) = \partial$, then $b\leq g_j[a_j]$ iff $g^\flat_i[b]\leq^\partial a_j$.
\end{itemize}
\end{enumerate}
It is routine to prove using the Lindenbaum-Tarski construction that $\mathbf{L}_\mathbb{DLE}^*$ is sound and complete with respect to.\ the class of all $\mathcal{L}_\mathrm{DLE}^*$-algebras. 
		
There two ways to specialize the language $\mc{L}_\mathbb{DLE}^*$ and hence the logic $\mbf{L}_\mathbb{DLE}$ to the strict implication language: a full and a partial specialization. The full specialization results a language of bi-intuitionsitic Lambek calculus $\mc{L}_\mrm{SI}^*$ which will not be explored in this paper. The partial specialization is to add the connectives $\{\bu, \imp, \limp\}$ to $\mc{L}_\mrm{SI}$ and get the language of full Lambek calculus, as we mentioned in the introduction, denoted by $\mc{L}_\mrm{LC}$. 
Clearly $\mc{L}_\mrm{SI}\subsetneq \mc{L}_\mrm{LC}\subsetneq\mc{L}_\mrm{SI}^*$. The partial specialization of $\mathcal{L}_\mathrm{DLE}^*$-algebras to the language $\mc{L}_\mrm{LC}$ is given in the following definition:

\begin{definition}
An algebra $\mf{A} = (A, \wedge,\vee,\top,\bot,\imp,\bu,\limp)$ is called a {\em bounded distributive lattice-ordered residuated groupoid} (BDRG), if $(A, \wedge,\vee,\top,\bot)$ is a bounded distributive lattice, and $\bu, \imp,\limp$ are binary operations on $A$ satisfying the following residuation law for all $a,b,c\in A$:
\begin{center}
(RES) $a\bu b\leq c$ iff $b\leq a\imp c$ iff $a\leq c\limp b$.
\end{center}
Let $\mathbb{BDRG}$ be the class of all BDRGs.
\end{definition}

\begin{definition}
The algebraic sequent calculus $\msf{BDFNL}$ consists of the following axioms and rules:
\begin{itemize}
\item Axioms:
\[
\mrm{(Id)}~\phi\vdash\phi,
\quad 
(\top)~\phi\vdash\top,
\quad
(\bot)~\bot\vdash\phi,
\]
\[
\mrm{(D)}~\phi\wedge(\psi\vee\gamma)\vdash(\phi\wedge\psi)\vee(\phi\wedge\gamma),
\]
\item Rules:
\[
(\wedge \mrm{L})~\frac{\phi_i\vdash\psi}{\phi_1\wedge\phi_2\vdash\psi}~(i=1,2),
\quad
(\wedge \mrm{R})~\frac{\gamma\vdash\phi\quad\gamma\vdash\psi}{\gamma\vdash\phi\wedge\psi},
\]
\[
(\vee \mrm{L})~\frac{\phi\vdash\gamma\quad\psi\vdash\gamma}{\phi\vee\psi\vdash\gamma},
\quad
(\vee \mrm{R})~\frac{\psi\vdash\phi_i}{\psi\vdash\phi_1\vee\phi_2}~(i=1,2),
\]
\[
(\mrm{Res1})~\frac{\phi\bu\psi\vdash\gamma}{\psi\vdash\phi\imp\gamma},
\quad
(\mrm{Res2})~\frac{\psi\vdash\phi\imp\gamma}{\phi\bu\psi\vdash\gamma},
\]
\[
(\mrm{Res3})~\frac{\phi\bu\psi\vdash\gamma}{\phi\vdash\gamma\limp\psi},
\quad
(\mrm{Res4})~\frac{\phi\vdash\gamma\limp\psi}{\phi\bu\psi\vdash\gamma},
\]
\[
(\mrm{cut})~\frac{\phi \vdash\psi\quad\psi\vdash\gamma}{\phi\vdash\gamma}.
\]
\end{itemize}
\end{definition}

\begin{fact}
The following monotonicity rules are derivable in $\msf{BDFNL}$:
\[
(1)\frac{\phi\vdash\psi}{\phi\bu\chi\vdash\psi\bu\chi},
\quad
(2)\frac{\phi\vdash\psi}{\chi\bu\phi\vdash\chi\bu\psi},
\]
\[
(3)\frac{\phi\vdash\psi}{\chi\imp\phi\vdash\chi\imp\psi},
\quad
(4)\frac{\phi\vdash\psi}{\psi\imp\chi\vdash\phi\imp\chi}.
\]
\end{fact}
\begin{proof}
Here we derive only (1) and (3). The remaining rules are derived similarly.
\[
\AxiomC{$\phi\vdash\psi$}
\AxiomC{$\psi\bu\chi\vdash\psi\bu\chi$}
\RightLabel{\small (Res3)}
\UnaryInfC{$\psi\vdash(\psi\bu\chi)\limp \chi$}
\RightLabel{\small (cut)}
\BinaryInfC{$\phi\vdash(\psi\bu\chi)\limp\chi$}
\RightLabel{\small (Res4)}
\UnaryInfC{$\phi\bu\chi\vdash\psi\bu\chi$}
\DisplayProof
\quad
\AxiomC{$\chi\imp\phi\vdash\chi\imp\phi$}
\RightLabel{\small (Res2)}
\UnaryInfC{$\chi\bu(\chi\imp\phi)\vdash\phi$}
\AxiomC{$\phi\vdash\psi$}
\RightLabel{\small (cut)}
\BinaryInfC{$\chi\bu(\chi\imp\phi)\vdash\psi$}
\RightLabel{\small (Res1)}
\UnaryInfC{$\chi\imp\phi\vdash\chi\imp\psi$}
\DisplayProof
\]
This completes the proof.
\qed
\end{proof}

The interpretation of $\mc{L}_\mrm{LC}$-sequents in BDRGs is standard, i.e., $\vdash$ is interpreted as the lattice order $\leq$. By $\mathbb{BDRG}\models\phi\vdash\psi$ we mean that $\phi\vdash\psi$ is valid in all BDRGs. An $\mc{L}_\mrm{LC}$-{\em supersequent} is an expression of the form $\Phi\Imp\chi\vdash\delta$ where $\Phi$ is a set of $\mc{L}_\mrm{LC}$-sequents.
We say that $\Phi\Imp\chi\vdash\delta$ is {\em derivable} in $\msf{BDFNL}$ if there exists a derivation of $\chi\vdash\delta$ from assumptions in $\Phi$. We say that $\Phi\Imp \chi\vdash\delta$ is {\em valid} in a BDRG $\mf{A}$ if $\mf{A}\models\Phi$ implies $\mf{A}\models\chi\vdash\psi$. We use $\mathbb{BDRG}\models\Phi\Imp\chi\vdash\delta$ to denote that $\chi\vdash\delta$ is valid in all BDRGs. By the Lindenbaum-Tarski construction, one gets the following result (cf.~\cite{Bus06}):

\begin{theorem}[strong completeness]
For every $\mc{L}_\mrm{LC}$-supersequent  $\Phi\Imp \chi\vdash\delta$, $\Phi\Imp\chi\vdash\delta$ is derivable in $\msf{BDFNL}$ if and only if $\mathbb{BDRG}\models \Phi\Imp\chi\vdash\delta$.
\end{theorem}

\subsection{Semantic conservativity via canonical extension}
In this subsection, we will present general results on the semantic conservativity of $\mc{L}^*_\mrm{DLE}$-logics over $\mc{L}_\mrm{DLE}$ logics. The proofs of the conservativity is by 
canonical extensions of DLEs. As a special case, the Lambek calculus $\msf{BDFNL}$ is a conservative extension of the strict implication logic $\mbf{S}_\mathbb{BDI}$.
First of all, let us recall some concepts from \cite{GH2001}. Given a bounded lattice $L$, a {\em completion} of $L$ is a complete lattice $C$ of which $L$ is a sublattice. For a completion $C$ of a lattice $L$, an element $x\in C$ is called {\em closed} if $x = \bigwedge_C F$ for some $F\sub L$; and $x\in C$ is called {\em open} if $x = \bigvee_C I$ for some $I\sub L$. The set of all closed elements in $C$ is denoted by $K({C})$, and the set of all open elements in $C$ by $O({C})$. A completion $C$ of a lattice $L$ is called 
\begin{itemize}
\item[--] {\em dense} if every element of $C$ can be represented both as a join of meets and as a meet of joins of elements from $L$.
\item[--] {\em compact} if for any $S\sub K({C})$ and $T\sub O({C})$, $\bigwedge S\leq \bigvee T$ iff there are finite subsets $S'\sub S$ and $T'\sub T$ with $\bigwedge S'\leq \bigvee T'$. 
\end{itemize}
A {\em canonical extension} of a lattice $L$ is a dense and compact completion of $L$. 
Every lattice has a canonical extension, denoted by $L^\delta$, which is unique up to an isomorphism \cite{GH2001}. 

A distributive lattice is {\em perfect} if it is complete, completely distributive and completely join-generated by the collection of its completely join-prime elements. Equivalently, a distributive lattice is perfect if and only if it is isomorphic to the lattice of upsets of some poset. A normal DLE is {\em perfect} if the underling distributive lattice is perfect, and each $f$-operation (respectively\ $g$-operation) is completely join-preserving (respectively\ meet-preserving) or completely meet-reversing (respectively\ join-reversing) in each coordinate.
It is well known that the canonical extension of a bounded distributive lattice is perfect (cf.\ e.g.\ \cite[Definition 2.14]{GeNaVe05}). 

Let $h: L\imp M$ be any map from a lattice $L$ to $M$. Following \cite[Definition 4.1]{GH2001}, one can define two maps $h^\sigma,h^\pi:L^\delta\imp M^\delta$ by setting: 
\begin{align*}
h^\sigma(u) &= \bigvee\{\bigwedge\{h(a) : a\in L~\&~x\leq a\leq y\} : K(L^\delta)\ni x\leq u\leq y \in O(L^\delta)\}.
\\
h^\pi(u) &= \bigwedge\{\bigvee\{h(a) : a\in L~\&~x\leq a\leq y\} : K(L^\delta)\ni x\leq u\leq y \in O(L^\delta)\}.
\end{align*}
Both $h^\sigma$ and $h^\pi$ extend $h$, and $h^\sigma\leq h^\pi$ pointwisely. In general, if $h$ is order-preserving, then $h^\sigma$ and $h^\pi$ are also order-preserving (\cite{GH2001}).
The canonical extension of an $\mathcal{L}_\mathrm{DLE}$-algebra $\mf{A} = (A, \mathcal{F}^A, \mathcal{G}^A)$ is the perfect  $\mathcal{L}_\mathrm{DLE}$-algebra
$\mf{A}^\delta = (A^\delta, \mathcal{F}^{A^\delta}, \mathcal{G}^{A^\delta})$ such that $f^{A^\delta}$ and $g^{A^\delta}$ are defined as the
$\sigma$-extension of $f^{A}$ and as the $\pi$-extension of $g^{A}$ respectively, for all $f\in \mathcal{F}$ and $g\in \mathcal{G}$.

\begin{lemma}\label{lemma:reduct-DLE}
For every $\mathcal{L}_\mathrm{DLE}^*$-algebra $(H,\wedge, \vee,\mathcal{F}^*, \mathcal{G}^*)$, its $(\wedge,\vee,\top,\bot, \mc{F}, \mc{G})$-reduct is a normal DLE.
\end{lemma}
\begin{proof}
Straightforward consequence of the fact that
left adjoints (respectively\ right adjoints) preserve existing joins (respectively\ meets). See \cite[Proposition 7.31]{DP02}.
\qed
\end{proof}

How can an $\mathcal{L}_\mathrm{DLE}$-algebra be extended to an $\mathcal{L}_\mathrm{DLE}^*$-algebra? This can be done in the canonical extension $\mf{A}^\delta = (A^\delta, \mathcal{F}^{A^\delta}, \mathcal{G}^{A^\delta})$ of $\mf{A}$. The canonical extension $A^\delta$ of the bounded distributive lattice $A$ is a perfect lattice which allows for defining adjoints. For each $f\in F^A$ and $1\leq i\leq n_f$, define
\[
f_i^\sharp[u_i] = 
\begin{cases}
\bigvee\{w\in A^\delta\mid f_i[w]\leq u_i\},~\mrm{if}~\varepsilon_f(i) = 1.\\
\bigwedge\{w\in A^\delta\mid f_i[w]\leq u_i\},~\mrm{if}~\varepsilon_f(i) = \partial.
\end{cases}
\]
For each $g\in G^A$ and $1\leq g\leq n_g$, define
\[
g_j^\flat[u_j] = 
\begin{cases}
\bigwedge\{w\in A^\delta\mid u_j\leq g_j[w]\},~\mrm{if}~\varepsilon_g(j) = 1.\\
\bigvee\{w\in A^\delta\mid u_j\leq g_j[w]\},~\mrm{if}~\varepsilon_g(j) = \partial.
\end{cases}
\]
Let ${\mc{F}^{A^\delta}}^*$ and ${\mc{G}^{A^\delta}}^*$ be extensions of $\mc{F}^{A^\delta}$ and $\mc{G}^{A^\delta}$ by adding all operators defined in the above way.

\begin{lemma}\label{lemma:ext-DLE}
The algebra ${\mf{A}^\delta}^E = (A^\delta, {\mc{F}^{A^\delta}}^*, {\mc{G}^{A^\delta}}^*)$ is a perfect $\mathcal{L}_\mathrm{DLE}^*$-algebra.
\end{lemma}
\begin{proof}
It suffices to show the residuation laws. We prove only the case for $f\in \mc{F}$ and $\varepsilon_f(i) = 1$. The remaining cases are similar. By definition, our goal is to show
\begin{center}
$f_i[u_i]\leq w$ iff $u_i\leq \bigvee\{v\in A^\delta\mid f_i[v]\leq w\}$.
\end{center}
The `only if' part is obvious. For the `if' part, assume $u_i\leq \bigvee\{v\in A^\delta\mid f_i[v]\leq w\}$. Then $f_i[u_i]\leq f_i[\bigvee\{v\in A^\delta\mid f_i[v]\leq w\}]$. By distributivity, one gets $f_i[u_i]\leq \bigvee\{f_i[v]\mid f_i[v]\leq w\}\leq w$.
\qed
\end{proof}

\begin{theorem}
\label{th:conservative extension}
The logic $\mathbf{L}_\mathbb{DLE}^*$ is a conservative extension of $\mathbf{L}_\mathbb{DLE}$, i.e., for every $\mathcal{L}_\mathrm{DLE}$-sequent $\phi\vdash\psi$, $\phi\vdash\psi$ is derivable in $\mathbf{L}_\mathbb{DLE}$ if and only if $\phi\vdash\psi$ is derivable in $\mathbf{L}_\mathbb{DLE}^*$.
\end{theorem}
\begin{proof}
Assume that $\phi\vdash\psi$ is derivable in $\mathbf{L}_\mathbb{DLE}$. By the completeness of $\mathbf{L}_\mathbb{DLE}$, $\phi\vdash\psi$ is valid in all DLEs.
By Lemma \ref{lemma:reduct-DLE}, 
$\phi\vdash\psi$ is also valid in all $\mc{L}_\mrm{DLE}^*$-algebras. Hence by the completeness of $\mathbf{L}_\mathbb{DLE}^*$, $\phi\vdash\psi$ is derivable in it. Conversely, assume that the $\mathcal{L}_\mathrm{DLE}$-sequent $\phi\vdash\psi$ is not derivable in $\mathbf{L}_\mathbb{DLE}$. Then by the completeness of $\mathbf{L}_\mathbb{DLE}$ with respect to.\ the class of DLEs, there exists a DLE $A$ and a variable assignment under which $\phi^A\not\leq \psi^A$, where $\phi^A$ and $\psi^A$ are values of $\phi$ and $\psi$ in $A$ under that assignment respectively. Consider the canonical extension $A^\delta$ of $A$.
Since $A$ is a subalgebra of $A^\delta$, the sequent $\phi\vdash\psi$ is not satisfied in $A^\delta$ under the variable assignment $\iota \circ v$ ($\iota$ denoting the canonical embedding $A\hookrightarrow A^\delta$). By Lemma \ref{lemma:ext-DLE}, one gets an $\mathcal{L}_\mathrm{DLE}^*$-algebra ${\mf{A}^\delta}^E$ which refutes $\phi\vdash\psi$. By the completeness of $\mathbf{L}_\mathbb{DLE}^*$, $\phi\vdash\psi$ is not derivable in $\mathbf{L}_\mathbb{DLE}^*$.
\qed
\end{proof}

The minimal logics $\mbf{L}_\mathbb{DLE}^*$ is
in the full language $\mc{L}^*_\mrm{DLE}$ with all adjoints. If the language $\mc{L}_\mrm{DLE}$ is expanded partially, i.e., with a portion of adjoint pairs, one can also obtain more general semantic conservativity results the proofs of which are the same as the proof of Theorem \ref{th:conservative extension}. Consider the language $\mc{L}_\mrm{DLE}(\mc{F}, \mc{G})$. Let $\mc{X}\sub \mc{F}$ and $\mc{Y}\sub\mc{G}$. Define $\mc{X}^\sharp$ as the extension of $\mc{X}$ with right adjoints, and $\mc{Y}^\flat$ as the extension of $\mc{Y}$ with left adjoints.

\begin{theorem}
\label{th:conservative extension 2}
Let $\mc{L}_\mrm{DLE}(\mc{F}, \mc{G})$ be a DLE-language, $\mc{X}\sub \mc{F}$ and $\mc{Y}\sub\mc{G}$.
The minimal logic $\mathbf{L}_\mathbb{DLE}^*(\mc{F}^*, \mc{G}^*)$ is a conservative extension of the minimal logic $\mathbf{L}_\mathbb{DLE}(\mc{F}, \mc{G}, \mc{X}^\sharp, \mc{Y}^\flat)$ which is also a conservative extension of the logic $\mbf{L}_\mathbb{DLE}(\mc{F}, \mc{G})$.
\end{theorem}

Let us consider the specialization of Theorem
\ref{th:conservative extension 2} to the strict implication logic $\mbf{S}_\mathbb{BDI}$ and the Lambek calculus $\msf{BDFNL}$.
First, as a corollary of Lemma \ref{lemma:reduct-DLE}, the $(\wedge, \vee, \bot, \top, \imp)$-reduct of a BDRG is a BDI.
Second, the canonical extension of 
a BDI $(A, \imp)$ is the $\pi$-extension $(A^\delta, \imp^\pi)$ which is also a BDI (cf.~\cite{GH2001,Mai14}), and 
we can define binary operators $\bu$ and $\limp$ on $A^\delta$ by setting
$u\bu v = \bigwedge\{w\in A^\delta\mid v\leq u\imp^\pi w\}$ and $u\limp v = \bigvee\{w\in A^\delta
\mid w\bu v\leq u\}$.
As a corollary of Lemma \ref{lemma:ext-DLE}, one gets the residuation law: for all $u, v, w\in A^\delta$,  $u\bu v\leq w$ iff $v\leq u\imp^\pi w$. Then one can apply Theorem \ref{th:conservative extension 2} immediately to get the following corollary:

\begin{corollary}\label{cor:bdfnl}
$\msf{BDFNL}$ is a conservative extension of $\mbf{S}_\mathbb{BDI}$.
\end{corollary}

\subsection{The algorithm $\msf{ALBA}$ for $\mathcal{L}_\mathrm{DLE}$-inequalities}
In this subsection, we will recall from \cite{GP2015} the definition of {\em inductive} $\mathcal{L}_\mathrm{DLE}$-inequalities on which the algorithm $\msf{ALBA}$ is guaranteed to succeed, and we will further specialize it to inequalities in the language of strict implication logic.
		
\begin{definition}[\textbf{Signed Generation Tree}]
\label{def: signed gen tree}
The \emph{positive} (respectively\ \emph{negative}) {\em generation tree} of any $\mathcal{L}_\mathrm{DLE}$-term $s$ is defined by labelling the root node of the generation tree of $s$ with the sign $+$ (respectively\ $-$), and then propagating the labelling on each remaining node as follows:
\begin{enumerate}
\item For any node labelled with $ \lor$ or $\land$, assign the same sign to its children nodes.
\item For any node labelled with $h\in \mathcal{F}\cup \mathcal{G}$ of arity $n_h\geq 1$, and for any $1\leq i\leq n_h$, assign the same (respectively\ the opposite) sign to its $i$th child node if $\varepsilon_h(i) = 1$ (respectively\ if $\varepsilon_h(i) = \partial$).
\end{enumerate}
Nodes in signed generation trees are \emph{positive} (respectively\ \emph{negative}) if they are signed $+$ (respectively\ $-$). The signed generation tree of an inequality $s\leq t$ consists of the generation trees of $+s$ and $-t$. 
\end{definition}
		
For any term (formula) $s(p_1,\ldots p_n)$, any order type $\varepsilon$ over $n$, and any $1 \leq i \leq n$, an \emph{$\varepsilon$-critical node} in a signed generation tree of $s$ is a leaf node $+p_i$ with $\varepsilon_i = 1$ or $-p_i$ with $\varepsilon_i = \partial$. An $\varepsilon$-{\em critical branch} in the tree is a branch from an $\varepsilon$-critical node. The intuition, which will be built upon later, is that variable occurrences corresponding to $\varepsilon$-critical nodes are \emph{to be solved for}, according to $\varepsilon$.
		
		For every term $s(p_1,\ldots p_n)$ and every order type $\varepsilon$, we say that $+s$ (respectively\ $-s$) {\em agrees with} $\varepsilon$, and write $\varepsilon(+s)$ (respectively\ $\varepsilon(-s)$), if every leaf in the signed generation tree of $+s$ (respectively\ $-s$) is $\varepsilon$-critical.
		In other words, $\varepsilon(+s)$ (respectively\ $\varepsilon(-s)$) means that all variable occurrences corresponding to leaves of $+s$ (respectively\ $-s$) are to be solved for according to $\varepsilon$. We will also write $+s'\prec \ast s$ (respectively\ $-s'\prec \ast s$) to indicate that the subterm $s'$ inherits the positive (respectively\ negative) sign from the signed generation tree $\ast s$. Finally, we will write $\varepsilon(\gamma) \prec \ast s$ (respectively\ $\varepsilon^\partial(\gamma_h) \prec \ast s$) to indicate that the signed subtree $\gamma$, with the sign inherited from $\ast s$, agrees with $\varepsilon$ (respectively\ with $\varepsilon^\partial$).
\begin{definition}
\label{def:good:branch}
Nodes in signed generation trees will be called \emph{$\Delta$-adjoints}, \emph{syntactically left residual (SLR)}, \emph{syntactically right residual (SRR)}, and \emph{syntactically right adjoint (SRA)}, according to the specification given in Table \ref{table:pia}.
A branch in a signed generation tree $\ast s$, with $\ast \in \{+, - \}$, is called a \emph{good branch} if it is the concatenation of two paths $P_1$ and $P_2$, one of which may possibly be of length $0$, such that $P_1$ is a path from the leaf consisting (apart from variable nodes) only of PIA-nodes, and $P_2$ consists (apart from variable nodes) only of Skeleton-nodes.\footnote{~These classes are grouped together into the super-classes \emph{Skeleton} and \emph{PIA} as indicated in the table. This organization is motivated and discussed in \cite{CFPS} and \cite{CoGhPa13} to establish a connection with analogous terminology in \cite{BeHovB12}.}
\begin{table}[\here]
\caption{Skeleton and PIA nodes for $\mathrm{DLE}$.}
\begin{center}
\begin{tabular}{| c | c |}
\hline
Skeleton &PIA\\
\hline
$\Delta$-adjoints & SRA \\
\begin{tabular}{ c c c c c c}
$+$ &$\vee$ &$\wedge$ &$\phantom{\lhd}$ & &\\
$-$ &$\wedge$ &$\vee$\\
\end{tabular}
&
\begin{tabular}{c c c c }
$+$ &$\wedge$ &$g$ & with $n_g = 1$ \\
$-$ &$\vee$ &$f$ & with $n_f = 1$ \\
\end{tabular}
\\
\hline
SLR &SRR\\
\begin{tabular}{c c c c }
$+$ & $\wedge$ &$f$ & with $n_f \geq 1$\\
$-$ & $\vee$ &$g$ & with $n_g \geq 1$ \\
\end{tabular}
&\begin{tabular}{c c c c}
$+$ &$\vee$ &$g$ & with $n_g \geq 2$\\
$-$ & $\wedge$ &$f$ & with $n_f \geq 2$\\
\end{tabular}
\\
\hline
\end{tabular}
\end{center}
\label{table:pia}
\vspace{-2em}
\end{table}
\end{definition}
		
\begin{definition}[Inductive inequalities]
\label{Inducive:Ineq:Def}
For any order type $\varepsilon$ and irreflexive and transitive relation $\Omega$ on $p_1,\ldots p_n$, the signed generation tree $*s$ $(* \in \{-, + \})$ of a term $s(p_1,\ldots p_n)$ is \emph{$(\Omega, \varepsilon)$-inductive} if
\begin{enumerate}
\item for all $1 \leq i \leq n$, every $\varepsilon$-critical branch with leaf $p_i$ is good (cf.\ Definition \ref{def:good:branch});
\item every $m$-ary SRR-node in the critical branch is of the form $ \circledast(\gamma_1,\dots,\gamma_{j-1},\beta,\gamma_{j+1}\ldots,\gamma_m)$, where for any $h\in\{1,\ldots,m\}\setminus j$: 
\begin{enumerate}
\item $\varepsilon^\partial(\gamma_h) \prec \ast s$ (cf.\ discussion before Definition \ref{def:good:branch}), and
\item $p_k <_{\Omega} p_i$ for every $p_k$ occurring in $\gamma_h$ and for every $1\leq k\leq n$.
\end{enumerate}
\end{enumerate}
We will refer to $<_{\Omega}$ as the \emph{dependency order} on the variables. An inequality $s \leq t$ is \emph{$(\Omega, \varepsilon)$-inductive} if the signed generation trees $+s$ and $-t$ are $(\Omega, \varepsilon)$-inductive. An inequality $s \leq t$ is \emph{inductive} if it is $(\Omega, \varepsilon)$-inductive for some $\Omega$ and $\varepsilon$.
\end{definition}

The definition of inductive inequalities for $\mc{L}_\mrm{DLE}$ can be easily specialized to the language $\mc{L}_\mrm{SI}$ of strict implication logic. The specialization needs only the classification of nodes in Table \ref{table:strict-nodes}.
\begin{table}[\here]
\label{table:strict-nodes}
\caption{Skeleton and PIA nodes for $\mc{L}_\mrm{SI}$.}
\begin{center}
\begin{tabular}{| c | c |}
\hline
Skeleton &PIA\\
\hline
$\Delta$-adjoints & SRA \\
\begin{tabular}{llllll}
$+$ &$\vee$ &$\wedge$ &$\phantom{\lhd}$ & &\\
$-$ &$\wedge$ & $\vee$\\
\end{tabular}
&
\begin{tabular}{llll}
$+$ &$\wedge$ & & \\
$-$ & $\vee$ &  & \\
\end{tabular}
\\
\hline
SLR &SRR\\
\begin{tabular}{llll}
$+$ & $\wedge$  & & \\
$-$ & $\vee$ & $\imp$ &  \\
\end{tabular}
&\begin{tabular}{llll}
$+$ & $\vee$ & $\imp$ &  \\
$-$ & $\wedge$ & &\\
\end{tabular}
\\
\hline
\end{tabular}
\end{center}
\vspace{-2em}
\end{table}

\begin{example}\label{exam:inductive}
Every sequent $\phi\vdash\psi$ can be presented as an inequality when $\vdash$ is replaced with $\leq$ due to the algebraic interpretation of $\vdash$. The inequalities obtained from Table \ref{table:cs} are inductive. For instance,
(Fr) is  inductive for $\varepsilon_p = \varepsilon_p= \varepsilon_r = 1$ and $p <_\Omega q <_\Omega r$. Henceforth we do not distinguish ``sequent" and ``inequality" if no confusion will arise.
\end{example}

Now we will define the algorithm $\msf{ALBA}$ in the setting of $\mathcal{L}_\mathrm{DLE}$. Consider the expanded language $\mathcal{L}_\mathrm{DLE}^{*+}$, which is built up on the base of the lattice constants $\top, \bot$ and a set of propositional variables $\mathsf{NOM}\cup \mathsf{CONOM}\cup \mathsf{AtProp}$ (the variables $\mathbf{i}, \mathbf{j}$ in $\mathsf{NOM}$ are referred to as {\em nominals}, and the variables $\mathbf{m}, \mathbf{n}$ in $\mathsf{CONOM}$ as {\em conomimals}), closing under the logical connectives of $\mathcal{L}_\mathrm{DLE}^*$. The natural semantic environment of $\mathcal{L}_\mathrm{DLE}^{*+}$ is given by perfect $\mathcal{L}_\mathrm{DLE}$-algebras. Let $A$ be a perfect 
$\mathcal{L}_\mathrm{DLE}$-algebra. An element $a\in A$ is {\em completely join-irreducible} (respectively completely meet-irreducible) if $a=\bigvee S$ (respectively $a=\bigwedge S$) implies that $a\in S$, for every subset $S$ of $A$. Nominals and conominals respectively range over the sets of the completely join-irreducible elements and the completely meet-irreducible elements of perfect DLEs.

An $\mathcal{L}_\mathrm{DLE}^{*+}$-inequality is an expression of the form $\phi\leq\psi$ where $\phi$ and $\psi$ are $\mathcal{L}_\mathrm{DLE}^{*+}$-formulas. An $\mathcal{L}_\mathrm{DLE}^{*+}$-quasi-inequality is an expression of the form $\phi_1\leq\psi_1~\&~\ldots~\&~\phi_n\leq\psi_n\Imp \phi_0\leq\psi_0$.
where all $\phi_i\leq\psi_i$ for $i\leq n$ are $\mathcal{L}_\mathrm{DLE}^{*+}$-inequalities. The algorithm $\msf{ALBA}$ manipulates inequalities and quasi-inequalities in $\mathcal{L}_\mathrm{DLE}^{*+}$.

The version of $\msf{ALBA}$ relative to $\mathcal{L}_\mathrm{DLE}$ runs as detailed in \cite{CoPa12,GP2015}. $\mathcal{L}_\mathrm{DLE}$-inequalities are equivalently transformed into the conjunction of one or more $\mathcal{L}_\mathrm{DLE}^{*+}$ quasi-inequalities, with the aim of eliminating propositional variable occurrences via the application of Ackermann rules. 
The proof of the soundness and invertibility of the general rules for the DLE-setting is similar to the one provided in \cite{CoPa12,CoGhPa13}. Here we recall the algorithm from \cite{GP2015} briefly.
The algorithm $\msf{ALBA}$ manipulates input inequalities $\phi\leq\psi$ and proceeds in three stages:
		
		\textbf{First stage: preprocessing and first approximation.} 
		$\msf{ALBA}$ preprocesses the input inequality $\phi\leq \psi$ by performing the following steps
exhaustively in the signed generation trees $+\phi$ and $-\psi$:
\begin{enumerate}
\item
\begin{enumerate}
\item Push down, towards variables, occurrences of $+\land$, by distributing each of them over their children nodes labelled with $+\lor$ which are not in the scope of PIA nodes;
\item Push down, towards variables, occurrences of $-\lor$, by distributing each of them over their children nodes labelled with $-\land$ which are not in the scope of PIA nodes;
\item Push down, towards variables, occurrences of $+f$ for any $f\in \mathcal{F}$, by distributing each such occurrence over its $i$th child node whenever the child node is labelled with $+\lor$ (respectively\ $-\land$) and is not in the scope of PIA nodes, and whenever $\varepsilon_f(i)=1$ (respectively\ $\varepsilon_f(i)=\partial$);				
\item Push down, towards variables, occurrences of $-g$ for any $g\in \mathcal{G}$, by distributing each such occurrence over its $i$th child node whenever the child node is labelled with $-\land$ (respectively\ $+\lor$) and is not in the scope of PIA nodes, and whenever $\varepsilon_g(i)=1$ (respectively\ $\varepsilon_g(i)=\partial$).
\end{enumerate}
\item Apply the splitting rules:
\[
\frac{\alpha\leq\beta\wedge\gamma}{\alpha\leq\beta\quad \alpha\leq\gamma}
\qquad
\frac{\alpha\vee\beta\leq\gamma}{\alpha\leq\gamma\quad \beta\leq\gamma}
\]
\item Apply the monotone and antitone variable-elimination rules:
\[
\frac{\alpha(p)\leq\beta(p)}{\alpha(\perp)\leq\beta(\perp)}
\qquad
\frac{\beta(p)\leq\alpha(p)}{\beta(\top)\leq\alpha(\top)}
\]
for $\beta(p)$ positive in $p$ and $\alpha(p)$ negative in $p$.
\end{enumerate}

Let $\mathsf{Preprocess}(\phi\leq\psi)$ be the finite set $\{\phi_i\leq\psi_i\mid 1\leq i\leq n\}$ of inequalities obtained after the exhaustive application of the previous rules. Next, the following {\em first approximation rule} is applied {\em only once} to every inequality in $\mathsf{Preprocess}(\phi\leq\psi)$:
\[
\frac{\phi\leq\psi}{\mbf{i}_0\leq\phi\ \ \ \psi\leq \mbf{m}_0}
\]
Here, $\mbf{i}_0$ and $\mbf{m}_0$ are a nominal and a conominal respectively. The first-approximation
		step gives rise to systems of inequalities $\{\mbf{i}_0\leq\phi_i, \psi_i\leq \mbf{m}_0\}$ for each inequality in $\mathsf{Preprocess}(\phi\leq\psi)$.

\textbf{Second stage: reduction-elimination cycle.} The goal of the reduction-elimination cycle is to eliminate all propositional variables from the systems
		received from the preprocessing phase. The elimination of each variable is effected by an
		application of one of the Ackermann rules given below. In order to apply an Ackermann rule, the
		system must have a specific shape. The adjunction, residuation, approximation, and splitting rules are used to transform systems into this shape.

\textbf{Residuation rules.} Here below we provide the residuation rules relative to each $f\in \mathcal{F}$ and $g\in \mathcal{G}$ of arity at least $1$: for each $1\leq j\leq n_f$ and each $1\leq k\leq n_g$:
\[
\AxiomC{$f_j[\psi_j]\leq \chi$}
\RightLabel{\footnotesize ($\varepsilon_f(j) = 1$),}
\UnaryInfC{$\psi_j\leq f_j^{\sharp}[\chi]$}
\DisplayProof
~~
\AxiomC{$f_j[\psi_j]\leq \chi$}
\RightLabel{\footnotesize ($\varepsilon_f(j) = \partial$),}
\UnaryInfC{$f_j^{\sharp}[\chi]\leq \psi_j$}
\DisplayProof
\]
\[
\AxiomC{$\chi \leq g_k[\psi_k]$}
\RightLabel{\footnotesize ($\varepsilon_g(k) = \partial$),}
\UnaryInfC{$\psi_k \leq g_k^{\flat}[\chi]$}
\DisplayProof
~~
\AxiomC{$\chi \leq g_k[\psi_k]$}
\RightLabel{\footnotesize ($\varepsilon_g(k) = 1$).}
\UnaryInfC{$g_k^{\flat}[\chi]\leq \psi_k$}
\DisplayProof
\]

\textbf{Approximation rules.} Here below we provide the approximation rules relative to each $f\in \mathcal{F}$ and $g\in \mathcal{G}$ of arity at least $1$: for each $1\leq j\leq n_f$ and each $1\leq k\leq n_g$,

\[
\AxiomC{$\mbf{i}\leq f_j[\psi_j]$}
\RightLabel{\footnotesize $(\varepsilon_f(j) = 1),$}
\UnaryInfC{$\mbf{i}\leq f_j[\mbf{j}]\quad\mbf{j}\leq \psi_j$}
\DisplayProof
\quad
\AxiomC{$g_k[\psi_k]\leq \mbf{m}$}
\RightLabel{\footnotesize $(\varepsilon_g(k) = 1),$}
\UnaryInfC{$g_k[\mbf{n}]\leq \mbf{m}\quad\psi_k\leq\mbf{n}$}
\DisplayProof
\]
\[
\AxiomC{$\mbf{i}\leq f_j[\psi_j]$}
\RightLabel{\footnotesize $(\varepsilon_f(j) = \partial),$}
\UnaryInfC{$\mbf{i}\leq f_j[\mbf{n}]\quad \psi_k\leq\mbf{n}$}
\DisplayProof
\quad
\AxiomC{$g_k[\psi_k]\leq \mbf{m}$}
\RightLabel{\footnotesize $(\varepsilon_g(k) = \partial),$}
\UnaryInfC{$g_k[\mbf{j}]\leq \mbf{m}\quad \mbf{j}\leq \psi_h$}
\DisplayProof
\]
where the variables $\mbf{i},\mbf{j}$ (respectively\ $\mbf{m},\mbf{n}$) are nominals (respectively\ conominals). The nominals and conominals introduced by approximation rules must be {\em fresh}, i.e.\ not occur in the system before applying the rule.
		
\textbf{Ackermann rules.} These rules are the core of $\msf{ALBA}$, since their application eliminates proposition variables. An important feature of Ackermann rules is that they are executed on the whole set of inequalities in which a given variable occurs, and not on a single inequality.
\begin{center}
\AxiomC{$\bigamp \{ \alpha_i \leq p \mid 1 \leq i \leq n \} \& \bigamp \{ \beta_j(p)\leq \gamma_j(p) \mid 1 \leq j \leq m \} \; \Rightarrow \; \mbf{i} \leq \mbf{m}$}
\RightLabel{(RAR)}
\UnaryInfC{$\bigamp \{ \beta_j(\bigvee_{i=1}^n \alpha_i)\leq \gamma_j(\bigvee_{i=1}^n \alpha_i) \mid 1 \leq j \leq m \} \; \Rightarrow \; \mbf{i} \leq \mbf{m}$}
\DisplayProof
\end{center}
where $p$ does not occur in $\alpha_1, \ldots, \alpha_n$, $\beta_{1}(p), \ldots, \beta_{m}(p)$ are positive in $p$, and $\gamma_{1}(p), \ldots, \gamma_{m}(p)$ are negative in $p$.
		
\begin{center}
\AxiomC{$\bigamp \{ p \leq \alpha_i \mid 1 \leq i \leq n \} \& \bigamp \{ \beta_j(p)\leq \gamma_j(p) \mid 1 \leq j \leq m \} \; \Rightarrow \; \mbf{i} \leq \mbf{m}$}
\RightLabel{(LAR)}
\UnaryInfC{$\bigamp \{ \beta_j(\bigwedge_{i=1}^n \alpha_i)\leq \gamma_j(\bigwedge_{i=1}^n \alpha_i) \mid 1 \leq j \leq m \} \; \Rightarrow \; \mbf{i} \leq \mbf{m}$}
\DisplayProof
\end{center}
where $p$ does not occur in $\alpha_1, \ldots, \alpha_n$, $\beta_{1}(p), \ldots, \beta_{m}(p)$ are negative in $p$, and $\gamma_{1}(p), \ldots, \gamma_{m}(p)$ are positive in $p$.
		
		\textbf{Third stage: output.}
		If there was some system in the second stage from which not all occurring propositional variables could be eliminated through the application of the reduction rules, then $\msf{ALBA}$ reports failure and terminates. Else, each system $\{\mbf{i}_0\leq\phi_i, \psi_i\leq \mbf{m}_0\}$ obtained from $\mathsf{Preprocess}(\varphi\leq \psi)$ has been reduced to a system, denoted $\mathsf{Reduce}(\varphi_i\leq \psi_i)$, containing no propositional variables. Let $\msf{ALBA}$$(\varphi\leq \psi)$ be the set of quasi-inequalities {\Large{\&}}$[\mathsf{Reduce}(\varphi_i\leq \psi_i) ]\Rightarrow \mbf{i}_0 \leq \mbf{m}_0$ for each $\varphi_i \leq \psi_i \in \mathsf{Preprocess}(\varphi\leq \psi)$.
Notice that all members of $\msf{ALBA}$$(\varphi\leq \psi)$ are free of propositional variables. $\msf{ALBA}$ returns $\msf{ALBA}$$(\varphi\leq \psi)$ and terminates. The proof of the following theorem is a straightforward generalization of \cite[Theorem 10.11]{CoPa12}, and hence its proof is omitted.

\begin{theorem}\label{Thm:ALBA:Success:Inductive}
For any  language $\mathcal{L}_\mathrm{DLE}$, its corresponding version of $\msf{ALBA}$ succeeds on all inductive $\mathcal{L}_\mathrm{DLE}$-inequalities, which are hence canonical\footnote{~An $\mathcal{L}_\mathrm{DLE}$-inequality $s\leq t$ is {\em canonical} if the class of $\mathcal{L}_\mathrm{DLE}$-algebras defined by $s\leq t$ is closed under canonical extension.} and their corresponding logics are complete with respect to\ the elementary classes of relational structures defined by their first-order correspondents.
\end{theorem}

For the specialization of the algorithm $\msf{ALBA}$ for $\mc{L}_\mrm{DLE}$ to the setting of strict implication logic, the only rules that need to note are the following residuation and approximation rules:
\begin{itemize}
\item[(a)] Residuation rule:
\[
\AxiomC{$\psi\leq\phi\imp\gamma$}
\UnaryInfC{$\phi\bu\psi\leq\gamma$}
\DisplayProof
\]
\item[(b)] Approximation rules:
\[
\AxiomC{$\phi\imp\psi\leq \mathbf{m}$}
\UnaryInfC{$\mathbf{i}\leq \phi \quad  \mathbf{i}\imp \psi\leq \mathbf{m}$}
\DisplayProof
\quad
\AxiomC{$\phi\imp\psi\leq \mathbf{m}$}
\UnaryInfC{$\psi\leq \mathbf{n} \quad  \phi\imp \mathbf{n}\leq \mathbf{m}$}
\DisplayProof
\]
\[
\AxiomC{$\mathbf{i}\leq \phi\bu\psi$}
\UnaryInfC{$\mathbf{j}\leq\phi\quad  \mathbf{i} \leq\mathbf{j}\bu \psi$}
\DisplayProof
\quad
\AxiomC{$\mathbf{i}\leq \phi\bu\psi$}
\UnaryInfC{$\mathbf{j}\leq \psi\quad \mathbf{i}\leq \phi\bu \mathbf{j}$}
\DisplayProof
\]
\end{itemize}

\begin{example}\label{example:alba}
The running of $\msf{ALBA}$ on the inductive $\mc{L}_\mrm{SI}$-sequents (inequalities) in Table \ref{table:cs} will produce pure inequalities as below:
\[
\begin{tabular}{|l|l|}
\hline
Sequent ~&~ Output\\
\hline
(I)~&~$\forall \mathbf{ij} (\mathbf{j}\bu \mathbf{i}\leq \mathbf{j})$\\
(Tr)~&~$\forall \mathbf{ij} (\mathbf{j}\bu \mathbf{i}\leq (\mathbf{j}\bu \mathbf{i})\bu \mathbf{i})$\\
(MP)~&~$\forall \mathbf{i} (\mathbf{i}\leq \mathbf{i}\bu \mathbf{i})$\\
(W)~&~$\forall \mathbf{ij} (\mathbf{i}\bu \mathbf{j}\leq \mathbf{j})$\\
(RT)~&~$\forall \mathbf{ijk} (\mathbf{i}\bu(\mathbf{j}\bu\mathbf{k})\leq \mathbf{i}\bu\mathbf{k})$\\
(B)~&~$\forall \mathbf{ijk} (\mathbf{i}\bu(\mathbf{j}\bu\mathbf{k})\leq (\mathbf{i}\bu\mathbf{j})\bu \mathbf{k})$\\
(B$'$)~&~$\forall \mathbf{ijk} (\mathbf{i}\bu(\mathbf{j}\bu\mathbf{k})\leq (\mathbf{i}\bu\mathbf{k})\bu \mathbf{j})$\\
(C)~&~$\forall \mathbf{ijk} (\mathbf{i}\bu(\mathbf{j}\bu\mathbf{k})\leq \mathbf{j}\bu(\mathbf{i}\bu\mathbf{k}))$\\
(Fr)~&~$\forall \mathbf{ijk} (\mathbf{i}\bu(\mathbf{j}\bu\mathbf{k})\leq (\mathbf{i}\bu\mathbf{j})\bu(\mathbf{i}\bu\mathbf{k}))$\\
(W$'$)~&~$\forall \mathbf{ij} (\mathbf{j}\bu \mathbf{i}\leq \mathbf{j}\bu (\mathbf{j}\bu \mathbf{i}))$\\
(Sym)~&~$\forall \mathbf{ij}\forall \mathbf{mn}(\mathbf{j}\bu\mathbf{i}\leq \mathbf{m}~\&~\mathbf{i}\imp\mathbf{n}\leq\mathbf{m}\Imp \mathbf{j}\leq\mathbf{m})$\\
(Euc)~&~$\forall \mathbf{ij}\forall \mathbf{m}\mathbf{n}_0\mathbf{n}_1(\mathbf{j}\bu\mathbf{i}\leq \mathbf{n}_0~\&~\mathbf{i}\imp\mathbf{n}_1\leq\mathbf{m}~\&~\mathbf{j}\imp\mathbf{n}_0\leq\mathbf{m}\Imp \top\leq\mathbf{m})$\\
(D)~&~$\top\imp\bot\leq\bot$\\
\hline
\end{tabular}
\]
Here we show only the running of $\msf{ALBA}$ on
$(p\imp q)\wedge (q\imp r)\leq p\imp r$ which proceeds as below:
\[
\begin{tabular}{lll}
&& $(p\imp q)\wedge (q\imp r)\leq p\imp r$ (First Approximation)\\
&$\Leftrightarrow$&
$\forall \mathbf{i}\forall \mathbf{m}(\mathbf{i}\leq (p\imp q)\wedge (q\imp r) ~\&~ p\imp r\leq \mathbf{m}\Imp \mathbf{i}\leq \mathbf{m})$ (Spliting)\\
&$\Leftrightarrow$&
$\forall \mathbf{i}\forall \mathbf{m}(\mathbf{i}\leq p\imp q~\&~\mathbf{i}\leq q\imp r ~\&~ p\imp r\leq \mathbf{m}\Imp \mathbf{i}\leq \mathbf{m})$
(Residuation)\\
&$\Leftrightarrow$&
$\forall \mathbf{i}\forall \mathbf{m}(p\bu \mathbf{i}\leq q~\&~q\bu \mathbf{i}\leq  r ~\&~ p\imp r\leq \mathbf{m}\Imp \mathbf{i}\leq \mathbf{m})$ (Approximation)\\
&$\Leftrightarrow$&
$\forall \mathbf{ij}\forall \mathbf{m}(p\bu \mathbf{i}\leq q~\&~q\bu \mathbf{i}\leq  r ~\&~\mathbf{j}\leq p~\&~\mathbf{j}\imp r\leq \mathbf{m}\Imp \mathbf{i}\leq \mathbf{m})$ (RAR)\\
&$\Leftrightarrow$&
$\forall\mathbf{ij}\forall \mathbf{m}(\mathbf{j}\bu \mathbf{i}\leq q~\&~q\bu \mathbf{i}\leq  r ~\&~ \mathbf{j}\imp r\leq \mathbf{m}\Imp \mathbf{i}\leq \mathbf{m})$ (RAR)\\
&$\Leftrightarrow$& $\forall \mathbf{ij}\forall \mathbf{m}((\mathbf{j}\bu \mathbf{i})\bu \mathbf{i}\leq  r ~\&~ \mathbf{j}\imp r\leq \mathbf{m}\Imp \mathbf{i}\leq \mathbf{m})$ (RAR)\\
&$\Leftrightarrow$&
$\forall \mathbf{ij}\forall \mathbf{m}(\mathbf{j}\imp ((\mathbf{j}\bu \mathbf{i})\bu \mathbf{i})\leq \mathbf{m}\Imp \mathbf{i}\leq \mathbf{m})$
\end{tabular}
\]
The output pure quasi-inequality is equivalent to $\forall \mathbf{ij} (\mathbf{j}\bu \mathbf{i}\leq (\mathbf{j}\bu \mathbf{i})\bu \mathbf{i})$. 
\end{example}

The algorithm $\msf{ALBA}$ for $\mc{L}_\mrm{DLE}$-logic 
described above does not only work for the distributive setting but also in general work for non-distributive lattice setting \cite{CoPa11}.
Hence the algorithm $\msf{ALBA}$ can be specialized to the full Lambek calculus. 
For the $\{\bu, \limp, \imp\}$-fragment of full Lambek calculus,
Kurtonina \cite{Kur94} presented a set of 
Sahlqvist formulas from which the first-order correspondents 
can be calculated by the Sahlqvist-van Benthem quantifier elimination procedure. Kurtonina's definition of Sahlqvist formulas is narrower than inductive inequalities provided by $\msf{ALBA}$. For example, The (Fr) inequality 
is inductive but not Sahlqvist. This remark is also discussed in \cite[Example 3.8]{CoPa11}.

\subsection{First-order correspondents}
Given an inductive $\mc{L}_\mrm{DLE}$-inequality 
$\phi\leq \psi$, the running of $\msf{ALBA}$ on it will output a pure quasi-inequality, namely, a quasi-inequality in which no propositional variable occurs. Then the first-order correspondent of $\phi\leq\psi$ is obtained when the Kripke semantics for $\mc{L}_\mrm{DLE}^{*+}$ is given such that $\mc{L}_\mrm{DLE}^{*+}$-terms are translated into a first-order language. For calculating the first-order correspondents of inductive $\mc{L}_\mrm{SI}$-inequalities, there are two kinds of Kripke semantics for the language $\mc{L}_\mrm{LC}^+$ (i.e., the extension of $\mc{L}_\mrm{LC}$ with normals and conominals): {\em binary} and {\em ternary} relational semantics.

{\em Binary relational semantics}.
The binary relational semantics for $\mc{L}_\mrm{LC}$ is given in ordinary Kripke structures. A {\em binary frame} is a 
pair $\mc{F}=(W, R)$ where $W$ is a non-empty set and $R$ is a binary relation on $W$. A {\em binary model} is a triple $\mc{M}=(W, R, V)$ where $(W, R)$ is a binary frame and $V:\msf{Prop}\cup\msf{NOM} \cup \msf{CONOM}\imp \mc{P}(W)$ is a valuation such that (i) for each $\mathbf{i}\in \msf{NOM}$, $V(\mathbf{i})=\{w\}$ for some $w\in W$; and (ii) for each $\mathbf{m}\in \msf{CONOM}$, $V(\mathbf{m})=W-\{u\}$ for some $u\in W$. Note that here there are no additional conditions assumed for the binary relation or the valuation.
For any $\mc{L}_\mrm{SI}$-formula $\phi$, the {\em satisfiability relation} $\mc{M},w\models\phi$ under the binary relational semantics is defined inductively as follows:
\begin{enumerate}
\item $\mc{M},w\models p$ iff $w\in V({p})$.
\item $\mc{M},w\models \mbf{i}$ iff $V(\mbf{i}) = \{w\}$.
\item $\mc{M},w\models \mbf{m}$ iff $V(\mbf{m}) = W-\{w\}$.
\item $\mc{M},w\not\models\bot$.
\item $\mc{M},w\models\phi\wedge\psi$ iff $\mc{M},w\models\phi$ and $\mc{M},w\models\psi$.
\item $\mc{M},w\models\phi\vee\psi$ iff $\mc{M},w\models\phi$ or $\mc{M},w\models\psi$.
\item $\mc{M},w\models\phi\imp\psi$ iff $\forall u\in W(wRu~\&~\mc{M},u\models\phi\Imp \mc{M},u\models\psi)$.
\item $\mc{M},w\models\phi\limp\psi$ iff $\forall u\in W(uRw~\&~\mc{M},u\models\psi\Imp \mc{M},w\models\phi)$.
\item $\mc{M},w\models \phi\bu\psi$ iff $\exists u\in W(uRw~\&~\mc{M},w\models \phi~\&~\mc{M},u\models \psi)$.
\end{enumerate}
Without the semantic clauses for nominals, conominals, $\limp$ and $\bu$, we get the binary relation semantics for strict implication language \cite{CJ03}.\footnote{~In \cite{CJ03}, the least weak strict implication logic $\msf{wK_\sigma}$ is introduced using sequents and shown to be strongly complete with respect to the class of all frames under the binary relational semantics. It is not hard to check that the algebraic sequent system $\mbf{S}_\mathbb{WH}$ is equivalent to $\msf{wK_\sigma}$.} The algorithm $\msf{ALBA}$ provides a general correspondence theory for the issue of the frame definability by sequents raised in \cite{CJ03}.

For a binary frame $\mc{F}=(W, R)$, the dual algebra of $\mc{F}$ is defined as $\mc{F}^+ = (\mc{P}(W), \cup, \cap, \emptyset, W, \imp_R^2, \bu_R^2, \limp_R^2)$ where $\imp_R^2$, $\limp_R^2$ and $\bu_R^2$ are binary operations defined on $\mc{P}(W)$ by setting
\begin{enumerate}
\item $X\imp_R^2 Y = \{w\in W\mid R(w)\cap X\sub Y\}$;
\item $X\limp_R^2 Y = \{w\in W\mid \forall u(uRw~\&~u\in Y\Imp w\in X)\}$;
\item $X\bu_R^2 Y = \{w\in W\mid \exists u(Ruw~\&~w\in X~\&~u\in Y)\}$;
\end{enumerate}
It is easy to prove that the algebra $\mc{F}^+$ is a BDRG. 
As \cite[Theorem 8.1]{CoPa12}, $\msf{ALBA}$ is also correct on binary relational frames. Then we can calculate the first-order correspondents of inductive $\mc{L}_\mrm{SI}$-sequents under the binary relational semantics.

\begin{example}\label{example:binary}
The outputs of $\msf{ALBA}$ running on the inductive inequalities in Example \ref{example:alba} 
can be transformed into first-order correspondents of the corresponding inductive sequents under the binary relational semantics as below:
\[
\begin{tabular}{|l|l|}
\hline
Sequent ~&~Binary Relational Correspondent\\
\hline
(I)~&~$\forall xy(Ryx \supset x=x)$\\
(Tr)~&~$\forall xy(Ryx \supset Ryx)$\\
(MP)~&~$\forall xRxx$\\
(W)~&~$\forall xy(Ryx\supset x=y)$\\
(RT)~&~$\forall xyz(Rxy\wedge Ryz\supset Rxz)$\\
(B)~&~$\forall xyz(Ryx\wedge Rzy \supset Rzx\wedge Ryx)$\\
(B$'$)~&~$\forall xyz(Ryx\wedge Rzy \supset Ryx\wedge Rzx)$\\
(C)~&~$\forall xyz(Ryx\wedge Rzy \supset Rxx\wedge x=y\wedge Rzx)$\\
(Fr)~&~$\forall xyz(Ryx\wedge Rzy \supset Rxx\wedge Ryx\wedge Rzx)$\\
(W$'$)~&~$\forall xy(Ryx\supset Rxx)$\\
(Sym)~&~$\forall xy(Rxy\supset Ryx)$\\
(Euc)~&~$\forall xyz(Rxy\wedge Rxz\supset Ryz)$\\
(D)~&~$\forall x\exists yRxy$\\
\hline
\end{tabular}
\]
Here we calculate only the first-order binary relational correspondents of
(Tr) and (Sym).

(1) The output of running $\msf{ALBA}$ on (Tr) is the pure inequality
$\forall \mathbf{ij} (\mathbf{j}\bu \mathbf{i}\leq (\mathbf{j}\bu \mathbf{i})\bu \mathbf{i})$. 
Note that $z\in \{x\}\bu^2\{y\}$ if and only if $Ryz$ and $z=x$. 
\begin{align*}
\forall \mathbf{ij} (\mathbf{j}\bu \mathbf{i}\leq (\mathbf{j}\bu\mathbf{i})\bu \mathbf{i})
\Leftrightarrow&~
\forall xy(\{x\}\bu^2 \{y\} \sub (\{x\}\bu^2 \{y\})\bu^2 \{y\})\\
\Leftrightarrow&~
\forall xyz(z\in \{x\}\bu^2\{y\} \supset z\in (\{x\}\bu^2 \{y\})\bu^2 \{y\})\\
\Leftrightarrow&~
\forall xyz(Ryz\land z=x \supset \exists u(Ruz\wedge z\in \{x\}\bu^2 \{y\} \wedge u=y))\\
\Leftrightarrow&~
\forall xyz(Ryz\land z=x \supset Ryz\wedge Ryz\wedge z=x)\\
\Leftrightarrow&~
\forall xyz(Ryz\land z=x \supset Ryz\wedge z=x)
\end{align*}
which is a tautology. (Tr) is in fact derivable in $\mbf{S}_\mathbb{WH}$, and the system $\mbf{S}_\mathbb{WH}$ is strongly complete with respect to the class of all binary frames (\cite{CJ03}).

(2) The output of running $\msf{ALBA}$ on (Sym) is the pure quasi-inequality
$\forall \mathbf{ij}\forall \mathbf{mn}(\mathbf{j}\bu\mathbf{i}\leq \mathbf{m}~\&~\mathbf{i}\imp\mathbf{n}\leq\mathbf{m}\Imp \mathbf{j}\leq\mathbf{m})$. Let $\mathbf{j}, \mathbf{i}, \mathbf{m}, \mathbf{n}$ be interpreted as $\{x\}, \{y\}, \{u\}^c, \{v\}^c$ respectively where $(.)^c$ is the complement operation. The calculation is as below:
\begin{align*}
\mathbf{j}\bu\mathbf{i}\leq \mathbf{m} \Leftrightarrow&~\{x\}\bu^2\{y\}\sub \{u\}^c\\
\Leftrightarrow&~\forall z(Ryz\wedge z=x\supset z\neq u)\\
\Leftrightarrow&~Ryx\supset x\neq u\\
\mbf{i}\imp\mbf{n}\leq\mbf{m} \Leftrightarrow&~ \forall w(w\in \{y\}\imp\{v\}^c\supset w\neq u)\\
\Leftrightarrow&~\forall w(\forall w_0(Rww_0\wedge w_0=y\supset w_0\neq v)\supset w\neq u)\\
\Leftrightarrow&~\forall w((Rwy\supset y\neq v)\supset w\neq u)\\
\Leftrightarrow&~\forall w(w=u\supset Rwy\wedge y=v)\\
\Leftrightarrow&~Ruy\wedge y=v\\
\forall \mathbf{ij}\forall \mathbf{mn}(\mathbf{j}\bu\mathbf{i}\leq \mathbf{m}&~\&~\mathbf{i}\imp\mathbf{n}\leq\mathbf{m}\Imp \mathbf{j}\leq\mathbf{m})\\
\Leftrightarrow&~\forall xyuv((Ryx\supset x\neq u)\wedge Ruy\wedge y=v\supset x\neq u)\\
\Leftrightarrow&~\forall xyu((Ryx\supset x\neq u)\wedge Ruy\supset x\neq u)\\
\Leftrightarrow&~\forall xyu(x=u\supset (Ruy\supset Ryx\wedge x=u))\\
\Leftrightarrow&~\forall xy(Rxy\supset Ryx)
\end{align*}
The sequent (Sym) defines the symmetry condition on binary frames.
\end{example}

{\em Ternary relational semantics}. The strict implication can be viewed as a binary modal operator added to distributive lattices, and hence there is a ternary relational semantics for it (cf.~\cite{BDV2001,DGP05}). A {\em ternary frame} is a frame $\mf{F} = (W, S)$ where $S$ is a ternary relation on $W$. A {\em ternary model} is a ternary frame with a valuation. The satisfiability relation $\mf{M}, w\Vdash\phi$ for the language $\mc{L}_\mrm{LC}$ under the ternary relational semantics is defined as usual. In particular, the semantic clauses for implications and the product are the following (cf.~\cite{Kur94}):
\begin{enumerate}
\item $\mf{M},w\Vdash \phi\imp\psi$ iff $\forall u,v(Svuw~\&~\mf{M},u\Vdash\phi~\Imp~\mf{M},v\Vdash\psi)$.
\item $\mf{M},w\Vdash \phi\limp\psi$ iff $\forall u,v(Svwu~\&~\mc{M},u\Vdash\psi~\Imp~\mf{M},v\Vdash\phi)$.
\item $\mf{M},w\Vdash \phi\bu\psi$ iff $\exists u,v(Swuv~\&~\mf{M},u\Vdash\phi~\&~\mf{M},v\Vdash\psi)$.
\end{enumerate}
Given a ternary frame $\mf{F}=(W, S)$, the dual of $\mf{F}$ is defined as $\mf{F}^* = (\mc{P}(W), \cup, \cap, \emptyset, W, \imp_S^3, \bu_S^3, \limp_S^3)$ where $\imp_S^3$, $\limp_S^3$ and $\bu_S^3$ are binary operations defined on $\mc{P}(W)$ by
\begin{enumerate}
\item $X\imp_S^3 Y = \{w\in W\mid \forall uv (Svuw~\&~u\in X\Imp v\in Y)\}$;
\item $X\limp_S^3 Y = \{w\in W\mid \forall uv(Svwu~\&~u\in Y\Imp v\in X\}$;
\item $X\bu_S^3 Y = \{w\in W\mid \exists uv(Swuv~\&~u\in X~\&~v\in Y)\}$.
\end{enumerate}
It is easy to check that $\mf{F}^*$ is a BDRG. Then under the ternary relational semantics one can calculate the first-order correspondents of inductive sequents.

\begin{example}
As Example \ref{example:binary}, we present the first-order correspondents of these inductive sequents under the ternary relational semantics as below:
\[
\begin{tabular}{|l|l|}
\hline
Sequent ~&~Ternary Relational Correspondent\\
\hline
(I)~&~$\forall xyz(Szxy \supset z=x)$\\
(Tr)~&~$\forall xyz(Szxy \supset \exists u(Szuy\wedge Suxy))$\\
(MP)~&~$\forall xSxxx$\\
(W)~&~$\forall xyz(Szyx\supset z=y)$\\
(RT)~&~$\forall xyzuv(Suxv\wedge Svyz\supset Suxz)$\\
(B)~&~$\forall xyzuw(Suxw\wedge Swyz\supset \exists v(Suvz\wedge Svxy))$\\
(B$'$)~&~$\forall xyzuw(Suxw\wedge Swyz\supset \exists v(Suvy\wedge Svxz))$\\
(C)~&~$\forall xyzuw(Suxw\wedge Swyz\supset \exists v(Suyv\wedge Svxz))$\\
(Fr)~&~$\forall xyzuw(Suxw\wedge Swyz\supset \exists v_0v_1(Suv_0v_1\wedge Sv_0xy\wedge Sv_1xz))$\\
(W$'$)~&~$\forall xyz(Suxy\supset \exists u(Szxu\wedge Suxy))$\\
(Sym)~&~$\forall xyv(Svyx\supset Sxxy)$\\
(Euc)~&~$\forall xyzuv(Suxz\wedge Suyz\supset Svxz)$\\
(D)~&~$\forall x\exists yzSzyx$\\
\hline
\end{tabular}
\]
Here we calculate only the first-order ternary relational correspondents of
(Tr) and (Sym). Note that $z\in \{x\}\bu^3\{y\}$ if and only if $Szxy$.
\begin{align*}
\forall \mathbf{ij} (\mathbf{j}\bu \mathbf{i}\leq (\mathbf{j}\bu\mathbf{i})\bu \mathbf{i})~
\Leftrightarrow~&~
\forall xy(\{x\}\bu^3 \{y\} \sub (\{x\}\bu^3 \{y\})\bu^3 \{y\})\\
\Leftrightarrow~&~
\forall xyz(z\in \{x\}\bu^3\{y\} \supset z\in (\{x\}\bu^3 \{y\})\bu^3 \{y\})\\
\Leftrightarrow~&~
\forall xyz(Szxy \supset \exists uv(Szuv\wedge u\in (\{x\}\bu^3 \{y\})\wedge v\in \{y\}))\\
\Leftrightarrow~&~
\forall xyz(Szxy \supset \exists uv(Szuv\wedge Suxy\wedge v\in \{y\}))\\
\Leftrightarrow~&~
\forall xyz(Szxy \supset \exists u(Szuy\wedge Suxy)).
\end{align*}
The result is not a tautology. The sequent (Tr) defines a special class of ternary relational frames.

(2) For (Sym), let $\mathbf{j}, \mathbf{i}, \mathbf{m}, \mathbf{n}$ be interpreted as $\{x\}, \{y\}, \{u\}^c, \{v\}^c$ respectively where $(.)^c$ is the complement operation. The calculation is as below:
\begin{align*}
\mathbf{j}\bu\mathbf{i}\leq \mathbf{m} \Leftrightarrow&~\{x\}\bu^3\{y\}\sub \{u\}^c\\
\Leftrightarrow&~\forall z(Szxy\supset z\neq u)\\
\Leftrightarrow&~\forall z(z=u\supset\ \sim Szxy)\\
\Leftrightarrow&~\sim Suxy\\
\mbf{i}\imp\mbf{n}\leq\mbf{m} \Leftrightarrow&~ \forall w(w\in \{y\}\imp\{v\}^c\supset w\neq u)\\
\Leftrightarrow&~\forall w(\forall w_0w_1(Sw_1w_0w\wedge w_0=y\supset w_1\neq v)\supset w\neq u)\\
\Leftrightarrow&~\forall w(\forall w_1(Sw_1yw\supset w_1\neq v)\supset w\neq u)\\
\Leftrightarrow&~\forall w(\forall w_1(w_1=v\supset\ \sim Sw_1yw)\supset w\neq u)\\
\Leftrightarrow&~\forall w(\ \sim Svyw \supset w\neq u)\\
\Leftrightarrow&~\forall w(w=u\supset Svyw)\\
\Leftrightarrow&~Svyu\\
\forall \mathbf{ij}\forall \mathbf{mn}(\mathbf{j}\bu\mathbf{i}\leq \mathbf{m}&~\&~\mathbf{i}\imp\mathbf{n}\leq\mathbf{m}\Imp \mathbf{j}\leq\mathbf{m})\\
\Leftrightarrow&~\forall xyuv(\sim Suxy\wedge Svyu \supset x\neq u)\\
\Leftrightarrow&~\forall xyuv(x=u\supset (Suxy\vee\sim Svyu))\\
\Leftrightarrow&~\forall xyv(Svyx\supset Sxxy)
\end{align*}
The sequent (Sym)
defines ternary frames satisfying $\forall xyv(Svyx\supset Sxxy)$.
\end{example}

\section{Algebraic correspondence: an 
application of $\msf{ALBA}$}

The algorithm $\msf{ALBA}$ is essentially a calculus for correspondence between non-classical logic and first-order logic.
It is used for obtaining analytic rules in display calculi for DLE-logics \cite{GP2015}.
For the main purpose of the present paper, 
we will use $\msf{ALBA}$ in a modified form, i.e., the Ackermann based calculus $\msf{ALC}$ based on $\msf{BDFNL}$, as a tool for obtaining analytic rules from certain axioms in the strict implication logic such that Gentzen-style cut-free sequent calculi will be constructed in the next section. 
The calculus $\msf{ALC}$ is also a calculus designed for correspondence, not correspondence between DLE-language and first-order language over Kripke frames, but correspondence over BDRGs between the language $\mc{L}_\mrm{SI}$ and the language $\mc{L}_\bu$ built from propositional variables and constants $\top, \bot$ using only the operator $\bu$ of product. The language $\mc{L}_\bu$ is quite natural because many properties of the product, e.g. the associativity, commutativity, contraction and weakening,  can be defined in terms of $\mc{L}_\bu$-sequents. 

Let us start from a motivating example.
The logic $\mbf{S}_\mathbb{WH}$ for weak Heyting algebras is obtained from $\mbf{S}_\mathbb{BDI}$ by adding the inductive sequents $\mrm{(Tr)}~(p\imp q)\wedge(q\imp r)\vdash p\imp r$ and $(\mrm{I})~q\vdash p\imp p$. Obviously, the logic $\mbf{S}_\mathbb{WH}$ can be conservatively extended to the extension of $\msf{BDFNL}$ with all instances of $(\mrm{Tr})$ and $(\mrm{I})$.
From proof-theoretic point of view, we need to know which structural rules the additional axioms can be equivalently transformed into if there exists.\footnote{~In \cite[Section 2.5]{Res94b}, some contraction rules are shown to guarantee certain axioms. For example, $(\mrm{I})$ follows from the weakening rule $X\cdot Y\Imp X$, and
 $(\mrm{Tr})$ follows from Restall's contraction rule $(\mrm{CSyll})~X;Y\Imp (X;Y);Y$.}
 In fact, in $\msf{BDFNL}$, one can prove that $(\mrm{I})$
 is equivalent to $(wl)~p\bu q\vdash p$, and that
 $(\mrm{Tr})$ is equivalent to $(tr)~p\bu s\vdash (p\bu s)\bu s$. Then it is easy to transform the sequents $(wl)$ and $(tr)$ into analytic rules as we will show in the next section. Here we are in fact saying that two sequents define the same class of BDRGS.
Formally, we say that a sequent $\phi\vdash\psi$ {\em algebraically corresponds} to $\phi'\vdash\psi'$ over BDRGs when they define the same class of BDRGs.

\begin{example}\label{exam:algcor}
The fact that the sequent $(\mrm{I})$ algebraically corresponds to $(wl)$ is follows immediately from the residuation law. Now we prove that the sequent $(\mrm{Tr})$ algebraically corresponds to $(tr)$. Let $\mf{A} = (A, \imp, \bu, \limp)$ be any BDRG. We need to show $\forall abd\in A[(a\imp b)\wedge(b\imp d)\leq a\imp d]$ iff $\forall ac\in A[a\bu c\leq (a\bu c)\bu c]$. One proof is as follows:
\[
\begin{tabular}{rll}
&&$\forall abd[(a\imp b)\wedge(b\imp d)\leq a\imp d]$\\
(I) &$\Leftrightarrow$&
$\forall abcd[c\leq a\imp b~\&~c\leq b\imp d\Imp c\leq a\imp d]$\\
(II) &$\Leftrightarrow$&
$\forall abcd[a\bu c\leq b~\&~b\bu c\leq d\Imp a\bu c\leq d]$\\
(III) &$\Leftrightarrow$&
$\forall acb[a\bu c\leq b\Imp a\bu c\leq b\bu c]$\\
(IV) &$\Leftrightarrow$&
$\forall ac[a\bu c\leq (a\bu c)\bu c]$.
\end{tabular}
\]
The steps (I) and (III) are obvious. The step (II) is by residuation in BDRGs. For the `if' part of step (IV),  assume that $\forall ac[a\bu c\leq (a\bu c)\bu c]$. Let $b\in A$ and $a\bu c\leq b$. Then one gets $(a\bu c)\bu c\leq b\bu c$. By the assumption, one gets $a\bu c\leq b\bu c$. The `only if' part is the instantiation of the universal quantifier.
\end{example}

For the algebraic correspondence, we will not take first-order language but $\mc{L}_\bu$ as the corresponding language of $\mc{L}_\mrm{SI}$. Nominals and conominals will not be needed. Instead, we introduce a calculus $\msf{ALC}$ in which propositional variables will play the role of nominals or comonimals in $\msf{ALBA}$.
The calculus $\msf{ALC}$ will be defined using supersequent rules of the form
\[
\frac{\Phi \Imp \phi \vdash\psi  }{\Phi'\Imp\phi'\vdash\psi'}{~({r}).}
\]
We say that $({r})$ is {\em valid} in $\mathbb{BDRG}$ if $\Phi'\Imp\phi'\vdash\psi'$ is valid in all BDRGs validating $\Phi\Imp\phi\vdash\psi$.

\begin{definition}
The Ackermann lemma based calculus $\msf{ALC}$ based on $\msf{BDFNL}$ consists of the following rules:
\begin{itemize}
\item[$(1)$] Splitting rules:
\[
\AxiomC{$\gamma\vdash\phi\wedge\psi,\Phi\Imp\chi\vdash\delta$}
\doubleLine
\RightLabel{\footnotesize ($\wedge$S)}
\UnaryInfC{$\gamma\vdash\phi,\gamma\vdash\psi,\Phi\Imp\chi\vdash\delta$}
\DisplayProof
\quad
\AxiomC{$\phi\vee\psi\vdash\gamma,\Phi\Imp\chi\vdash\delta$}
\doubleLine
\RightLabel{\footnotesize ($\vee$S)}
\UnaryInfC{$\phi\vdash\gamma,\psi\vdash\gamma,\Phi\Imp\chi\vdash\delta$}
\DisplayProof
\]
\item[$(2)$] Residuation rules:
\[
\AxiomC{$\psi\vdash\phi\imp\gamma,\Phi\Imp\chi\vdash\delta$}
\doubleLine
\RightLabel{\footnotesize (RL1)}
\UnaryInfC{$\phi\bu\psi\vdash\gamma,\Phi\Imp\chi\vdash\delta$}
\DisplayProof
\quad
\AxiomC{$\phi\vdash\gamma\limp\psi,\Phi\Imp\chi\vdash\delta$}
\doubleLine
\RightLabel{\footnotesize (RL2)}
\UnaryInfC{$\phi\bu\psi\vdash\gamma,\Phi\Imp\chi\vdash\delta$}
\DisplayProof
\]
\[
\AxiomC{$\Phi\Imp\psi\vdash\phi\imp\gamma$}
\doubleLine
\RightLabel{\footnotesize (RR1)}
\UnaryInfC{$\Phi\Imp\phi\bu\psi\vdash\gamma $}
\DisplayProof
\quad
\AxiomC{$\Phi\Imp\phi\vdash\gamma\limp\psi$}
\doubleLine
\RightLabel{\footnotesize (RR2)}
\UnaryInfC{$\Phi\Imp\phi\bu\psi\vdash\gamma$}
\DisplayProof
\]
\item[$(3)$] Approximation rules:
\[
\AxiomC{$\Phi\Imp\phi\vdash\psi$}
\doubleLine
\RightLabel{\footnotesize (Ap1)}
\UnaryInfC{$p\vdash\phi,\Phi\Imp p\vdash \psi$}
\DisplayProof
\quad
\AxiomC{$\Phi\Imp\phi\vdash\psi$}
\doubleLine
\RightLabel{\footnotesize (Ap2)}
\UnaryInfC{$\psi\vdash p,\Phi\Imp \phi\vdash p$}
\DisplayProof
\]
\[
\AxiomC{$\phi\imp\psi\vdash\gamma,\Phi\Imp\chi\vdash\delta$}
\doubleLine
\RightLabel{\footnotesize ($\imp$Ap1)}
\UnaryInfC{$p\vdash\phi,p\imp\psi\vdash\gamma,\Phi\Imp \chi\vdash \delta$}
\DisplayProof
\quad
\AxiomC{$\phi\imp\psi\vdash\gamma,\Phi\Imp\chi\vdash\delta$}
\doubleLine
\RightLabel{\footnotesize ($\imp$Ap2)}
\UnaryInfC{$\psi\vdash p, \phi\imp p\vdash \gamma,\Phi\Imp\chi\vdash\delta$}
\DisplayProof
\] 
\[
\AxiomC{$\gamma\vdash\phi\imp\psi,\Phi\Imp\chi\vdash\delta$}
\doubleLine
\RightLabel{\footnotesize ($\imp$Ap3)}
\UnaryInfC{$\phi\vdash p,\gamma\vdash p\imp\psi,\Phi\Imp \chi\vdash \delta$}
\DisplayProof
\quad
\AxiomC{$\gamma\vdash\phi\imp\psi,\Phi\Imp\chi\vdash\delta$}
\doubleLine
\RightLabel{\footnotesize ($\imp$Ap4)}
\UnaryInfC{$p\vdash \psi,\gamma\vdash\phi\imp p,\Phi\Imp \chi\vdash\delta$}
\DisplayProof
\] 
\[
\AxiomC{$\phi\vdash\psi\bu\gamma,\Phi\Imp\chi\vdash\delta$}
\doubleLine
\RightLabel{\footnotesize ($\bu$Ap1)}
\UnaryInfC{$p\vdash\psi,\phi\vdash p\bu\gamma,\Phi\Imp \chi\vdash \delta$}
\DisplayProof
\quad
\AxiomC{$\phi\vdash\psi\bu\gamma,\Phi\Imp\chi\vdash\delta$}
\doubleLine
\RightLabel{\footnotesize ($\bu$Ap2)}
\UnaryInfC{$p\vdash \gamma, \phi \vdash \psi\bu p,\Phi\Imp\chi\vdash\delta$}
\DisplayProof
\]
\[
\AxiomC{$\phi\bu\psi\vdash\gamma,\Phi\Imp\chi\vdash\delta$}
\doubleLine
\RightLabel{\footnotesize ($\bu$Ap3)}
\UnaryInfC{$\phi\vdash p, p\bu\psi\vdash\gamma,\Phi\Imp \chi\vdash \delta$}
\DisplayProof
\quad
\AxiomC{$\phi\bu\psi\vdash\gamma,\Phi\Imp\chi\vdash\delta$}
\doubleLine
\RightLabel{\footnotesize ($\bu$Ap4)}
\UnaryInfC{$\psi\vdash p, \phi\bu p\vdash\gamma,\Phi\Imp \chi\vdash \delta$}
\DisplayProof
\]
\[
\AxiomC{$\phi\wedge\psi\vdash\gamma,\Phi\Imp\chi\vdash\delta$}
\doubleLine
\RightLabel{\footnotesize ($\wedge$Ap5)}
\UnaryInfC{$\phi\vdash p, p\wedge\psi\vdash\gamma,\Phi\Imp \chi\vdash \delta$}
\DisplayProof
\quad
\AxiomC{$\phi\wedge\psi\vdash\gamma,\Phi\Imp\chi\vdash\delta$}
\doubleLine
\RightLabel{\footnotesize ($\wedge$Ap6)}
\UnaryInfC{$\psi\vdash p, \phi\wedge p\vdash\gamma,\Phi\Imp \chi\vdash \delta$}
\DisplayProof
\]
\[
\AxiomC{$\phi\vdash\psi\vee\gamma,\Phi\Imp\chi\vdash\delta$}
\doubleLine
\RightLabel{\footnotesize ($\vee$Ap1)}
\UnaryInfC{$p\vdash\psi,\phi\vdash p\vee\gamma,\Phi\Imp \chi\vdash \delta$}
\DisplayProof
\quad
\AxiomC{$\phi\vdash\psi\vee\gamma,\Phi\Imp\chi\vdash\delta$}
\doubleLine
\RightLabel{\footnotesize ($\vee$Ap2)}
\UnaryInfC{$p\vdash \gamma, \phi \vdash \psi\vee p,\Phi\Imp\chi\vdash\delta$}
\DisplayProof
\]
where $p$ is a fresh variable, i.e., a variable which does not occur in previous derivation.

\item[]$(4)$ Ackermann rules:
\[
\AxiomC{$\phi_1\vdash p,\ldots,\phi_n \vdash p, \Phi, \Phi'\Imp \chi\vdash\delta$}
\doubleLine
\RightLabel{\footnotesize (RAck)}
\UnaryInfC{$\Phi[\bigvee_{i=1}^n\phi_i/p],\Phi'\Imp(\chi\vdash\delta)^*$}
\DisplayProof
\]
where (i) $p$ does not occur in $\Phi'$ or $\phi_i$ for $1\leq i\leq n$; (ii) $\Phi=\{\psi_j\vdash\gamma_j\mid \psi_j(+p),\gamma_j(-p),1\leq j\leq m\}$ and $\Phi[\bigvee_{i=1}^n\phi_i/p] = \{\psi_j[\bigvee_{i=1}^n\phi_i/p]\vdash\gamma_j[\bigvee_{i=1}^n\phi_i/p]\mid\psi_j\vdash\gamma_j\in\Phi\}$; and (iii) either $p$ does not occur in $\chi\vdash\delta$ and $(\chi\vdash\delta)^* = \chi\vdash\delta$, or $\chi\vdash\delta$ is positive in $p$ and $(\chi\vdash\delta)^* = \chi[\bigvee_{i=1}^n\phi_i/p]\vdash \delta[\bigvee_{i=1}^n\phi_i/p]$.
\[
\AxiomC{$p\vdash\phi_1 ,\ldots,p\vdash\phi_n, \Phi, \Phi'\Imp \chi\vdash\delta$}
\doubleLine
\RightLabel{\footnotesize (LAck)}
\UnaryInfC{$\Phi[\bigwedge_{i=1}^n\phi_i/p],\Phi'\Imp(\chi\vdash\delta)^*$}
\DisplayProof
\]
where (i) $p$ does not occur in $\Phi'$ or $\phi_i$ for $1\leq i\leq n$; (ii) $\Phi=\{\psi_j\vdash\gamma_j\mid \psi_j(-p),\gamma_j(+p),1\leq j\leq m\}$ and $\Phi[\bigwedge_{i=1}^n\phi_i/p] = \{\psi_j[\bigwedge_{i=1}^n\phi_i/p]\vdash\gamma_j[\bigwedge_{i=1}^n\phi_i/p]\mid\psi_j\vdash\gamma_j\in\Phi\}$; and (iii) either $p$ does not occur in $\chi\vdash\delta$ and $(\chi\vdash\delta)^* = x\vdash\delta$, or $\chi\vdash\delta$ is negative in $p$ and $(\chi\vdash\delta)^* = \chi[\bigwedge_{i=1}^n\phi_i/p]\vdash \delta[\bigwedge_{i=1}^n\phi_i/p]$.
\end{itemize}

The double line in above rules means that the above and the below supersequents can be derived from each other.
A supersequent rule $({r})$ is said to be {\em derivable} in $\msf{ALC}$ if there is a derivation of the conclusion from the premiss of $({r})$ using only rules in $\msf{ALC}$.
\end{definition}

\begin{theorem}[Correctness]
All rules in $\msf{ALC}$ are valid in $\mathbb{BDRG}$.
\end{theorem}
\begin{proof}
The proof is routine. For details, see e.g. \cite{CoPa12}.
\qed
\end{proof}

Given a set of $\mc{L}_\mrm{SI}$-sequents $\Phi$, let $\msf{Alg}(\Phi)$ and $\msf{Alg}^+(\Phi)$ be the class of all BDIs 
and the class of all BDRGs validating all sequents in $\Phi$ respectively. 
Similarly, given a set of $\mc{L}_\bu$-sequents $\Psi$, let $\msf{Alg}^+(\Psi)$ be the class of all BDRGs validating all sequents in $\Psi$. Obviously, an $\mc{L}_\mrm{SI}$-sequent $\phi\vdash\psi$ corresponds to an $\mc{L}_\bu$-sequent $\phi'\vdash\psi'$ over BDRGs if and only if $\msf{Alg}^+(\phi\vdash\psi) = \msf{Alg}^+(\phi'\vdash\psi')$. 

\begin{proposition}\label{fact:cor}
Given an $\mc{L}_\mrm{SI}$-sequent $\phi\vdash\psi$ and an $\mc{L}_\bu$-sequent $\chi\vdash\delta$, if the rule
\[
\frac{\Imp \phi\vdash\psi}{\Imp \chi\vdash\delta}(r)
\]
is derivable in $\msf{ALC}$, then $\phi\vdash\psi$ algebraically corresponds to $\chi\vdash\delta$ over BDRGs.
\end{proposition}
\begin{proof}
Assume that the rule ($r$) is derivable in $\msf{ALC}$. 
By the correctness of $\msf{ALC}$, the premiss $\phi\vdash\psi$ and the conclusion $\chi\vdash\delta$ defines the same BDRGs, i.e., $\msf{Alg}^+(\phi\vdash\psi) = \msf{Alg}^+(\phi'\vdash\psi')$.
\qed
\end{proof}

By Proposition \ref{fact:cor}, one obtains a proof-theoretic tool for algebraic correspondence over BDRGs between the languages $\mc{L}_\mrm{SI}$ and $\mc{L}_\bu$.

\begin{example}\label{exam:correspondents}
Some $\mc{L}_\mrm{SI}$-sequents (inequalities) in Table \ref{table:cs} and their algebraic correspondents in $\mc{L}_\bu$ are listed in Table \ref{table:cp}.
\begin{table}[htbp]
\caption{Some Algebraic Correspondents}
\vspace{-1em}
\[
\begin{tabular}{|ll|ll|}
\hline
&$\mc{L}_\mrm{SI}$-sequent & & $\mc{L}_\bu$-sequent
\\
\hline
(I) & $q\vdash p\imp p$ & $(wl)$ & $p\bu q \vdash p$\\
(Tr) & $(p\imp q)\wedge (q\imp r)\vdash p\imp r$ & $(tr)$ & $p\bu s\vdash (p\bu s)\bu s$\\
(MP) & $p\wedge (p\imp q)\vdash q$ & $(ct)$ & $p\vdash p\bu p$\\
(W) & $p\vdash q\imp p$ & $(wr)$ & $q\bu p \vdash p$\\
(RT) & $p\imp q \vdash  r\imp (p\imp q)$ & $(rt)$ & $p\bu (r\bu s)\vdash p\bu s$\\
(B) & $p \imp q\vdash  (r\imp p)\imp (r\imp q)$& $(b)$ & $r\bu (s\bu t)\vdash (r\bu t)\bu s$\\
(B$'$) & $p\imp q\vdash (q\imp r)\imp (p\imp r)$ & $(b')$ & $p\bu (t\bu s)\vdash (p\bu s)\bu t$ \\
({C}) & $p\imp (q\imp r)\vdash q\imp (p\imp r)$ & $({c})$ & $p\bu(q\bu s)\vdash q\bu (p\bu s)$\\
(Fr) & $p\imp (q\imp r)\vdash (p\imp q)\imp (p\imp r)$ & $(fr)$ & $p\bu(u\bu s)\vdash (p\bu u)\bu (p\bu s)$\\
(W$'$) & $p\imp (p\imp q)\vdash p\imp q$ & ($w'$) & $p\bu r \vdash p\bu (p\bu r)$\\
\hline
\end{tabular}
\]
\vspace{-1em}
\label{table:cp}
\end{table}

$(\mrm{Tr})$ One proof is as follows:
\[
\AxiomC{$\Imp (p\imp q)\wedge (q\imp r)\vdash (p\imp r)$}
\RightLabel{\footnotesize (AAp1)}
\UnaryInfC{$s\vdash (p\imp q)\wedge (q\imp r)\Imp s\vdash p\imp r$}
\RightLabel{\footnotesize ($\wedge$S)}
\UnaryInfC{$s\vdash p\imp q, s\vdash q\imp r \Imp s\vdash p\imp r$}
\RightLabel{\footnotesize (RL1, RR1)}
\UnaryInfC{$p\bu s\vdash q, q\bu s\vdash r\Imp p\bu s\vdash r$}
\RightLabel{\footnotesize (RAck)}
\UnaryInfC{$p\bu s\vdash q\Imp p\bu s\vdash q\bu s$}
\RightLabel{\footnotesize (RAck)}
\UnaryInfC{$\Imp p\bu s\vdash (p\bu s)\bu s$}
\DisplayProof
\]
Other pairs of corresponding sequents can be proved similarly. See Appendix \ref{appendix:cp}.
\end{example}

\begin{remark}
Some inductive $\mc{L}_\mrm{SI}$-sequents have algebraic correspondents in $\mc{L}_\bu$ using $\msf{ALC}$. But it is not clear whether all inductive sequents in $\mc{L}_\mrm{SI}$ have algebraic correspondents in $\mc{L}_\bu$. Consider the sequents (Sym), (Euc) and (D). Our conjecture is that these sequents never correspond to any $\mc{L}_\bu$-sequents. Conversely, we conjecture that not all $\mc{L}_\bu$-sequents have their algebraic correspondents in $\mc{L}_\mrm{SI}$. Consider the inverse of $(tr)$ in Table \ref{table:cp}. We start from $(p\bu s)\bu s\vdash p\bu s$ and apply $\msf{ALC}$. The first step is to use approximation rule, and we get
\[
p\bu s\vdash q\Imp (p\bu s)\bu s\vdash q
\]
Using residuation rules, we get
\[
s\vdash p\imp q\Imp s\vdash (p\bu s)\imp q
\]
The next step is to consider using the left Ackermann rule because the term $p\bu s$ on the right hand side takes a negative position. Then we have
\[
t\leq p\bu s, s\vdash p\imp q\Imp s\vdash t\imp q
\]
Then there is no way to continue $\msf{ALC}$. It is rather likely that the sequent $(p\bu s)\bu s\vdash p\bu s$ has no algebraic correspondent in $\mc{L}_\mrm{SI}$. The general question on the expressive power of $\msf{ALC}$ will be explored in future work.
\end{remark}

Let $\Phi$ be a set of $\mc{L}_\mrm{SI}$-sequents and $\Psi$ a set of $\mc{L}_\bu$-sequents. We use the notation $\Phi\equiv_\msf{ALC}\Psi$ to denote that $\Psi$
consists of $\mc{L}_\bu$-sequents obtained 
from sequents in $\Phi$ using $\msf{ALC}$. Let $\mbf{S}_\mathbb{BDI}(\Phi)$ be the algebraic sequent system obtained from $\msf{S}_\mathbb{BDI}$ by adding all instances of sequents in $\Phi$ as axioms.
Similarly, let $\msf{BDFNL}(\Psi)$ be the algebraic sequent system obtained from $\msf{BDFNL}$ by adding all instances of sequents in $\Psi$ as axioms.
Clearly $\mbf{S}_\mathbb{BDI}(\Phi)$ is sound and complete with respect to $\msf{Alg}(\Phi)$, and $\msf{BDFNL}(\Psi)$ is sound and complete with respect to $\msf{Alg}^+(\Psi)$.

\begin{lemma}\label{lemma:reduct1}
Let $\Phi$ be a set of inductive $\mc{L}_\mrm{SI}$-sequents and $\Psi$ a set of $\mc{L}_\bu$-sequents. Assume $\Phi\equiv_\msf{ALC}\Psi$. For every algebra $\mf{A}$ in $\msf{Alg}^+(\Psi)$, its $(\wedge, \vee, \bot, \top, \imp)$-reduct is an algebra in $\msf{Alg}(\Phi)$.
\end{lemma}
\begin{proof}
Let $\mf{A}\in \msf{Alg}^+(\Psi)$. Then $\mf{A}\models \Psi$. By $\Phi\equiv_\msf{ALC}\Psi$, one gets $\mf{A}\models \Phi$. Hence the $(\wedge, \vee, \bot, \top, \imp)$-reduct of $\mf{A}$ is an algebra in $\msf{Alg}(\Phi)$.
\qed
\end{proof}

\begin{lemma}\label{lemma:extension1}
Let $\Phi$ be a set of inductive $\mc{L}_\mrm{SI}$-sequents and $\Psi$ a set of $\mc{L}_\bu$-sequents. Assume $\Phi\equiv_\msf{ALC}\Psi$. 
For every algebra $\mf{A} = (A, \imp)$ in $\msf{Alg}(\Phi)$, its canonical extension $\mf{A}^\delta = (A^\delta, \imp^\pi, \bu, \limp)$ is in $\msf{Alg}^+(\Psi)$.
\end{lemma}
\begin{proof}
Obviously, $\mf{A}^\delta$ is a BDRG. Moreover, $(A^\delta, \imp^\pi)\in \msf{Alg}(\Phi)$ because every inductive sequent in $\Phi$ is canonical. By $\Phi\equiv_\msf{ALC}\Psi$, one gets $\mf{A}^\delta\models \Psi$.
\qed
\end{proof}

By Lemma \ref{lemma:reduct1} and Lemma \ref{lemma:extension1}, one gets the following theorem immediately:

\begin{theorem}\label{thm:con2}
Let $\Phi$ be a set of inductive sequents in $\mc{L}$ and $\Psi$ a set of $\mc{L}_\bu$-sequents. Assume $\Phi\equiv_\msf{ALC}\Psi$. The algebraic sequent $\msf{BDFNL}(\Psi)$ is a conservative extension of $\msf{S}_\mathbb{BDI}(\Phi)$.
\end{theorem}

\begin{example}
Notice that $(\mrm{I})~q\vdash p\imp p$ corresponds to $(wl)~p\bu q\vdash p$, and $(\mrm{Tr})~(p\imp q)\wedge (q\imp r)\vdash p\imp r$ corresponds to $(tr)~p\bu s\vdash (p\bu s)\bu s$. Both $(\mrm{I})$ and $(\mrm{Tr})$ are inductive sequents. The algebras defined by $(wl)$ and $(tr)$ are BDRGs satisfying the conditions:
$(wl)~ a\bu b\leq a$ and $(tr)~a\bu b\leq (a\bu b)\bu b$. We call such algebras {\em residuated weak Heyting algebras}, and the class of such algebras is denoted by $\mathbb{RWH}$. By Theorem \ref{thm:con2}, the algebraic sequent system $\mbf{S}_\mathbb{RWH}$ is a conservative extension of $\mbf{S}_\mathbb{WH}$. For sequents in Example \ref{exam:correspondents}, one can get similar conservativity results.
\end{example}

\section{Gentzen-style sequent calculi}
In this section, we will first introduce a Gentzen-style cut-free sequent calculus $\msf{G}_\msf{BDFNL}$ for $\msf{BDFNL}$\footnote{~The sequent system for $\msf{BDFNL}$ defined in e.g. \cite{Bus06,BusF09} does not admit cut elimination. When the distributivity is added as an axiom, the sequent $\phi\wedge(\psi\vee (\chi\vee\delta))\vdash (\phi\wedge\psi)\vee ((\phi\vee\chi)\vee(\phi\wedge\delta))$ cannot be proved without cut.}, which will be presented by introducing structure operators separately for connectives $\wedge$ and $\bu$. By the conservativity of $\msf{BDFNL}$ over $\msf{S_{BDI}}$, and the subformula property of 
$\msf{G}_\msf{BDFNL}$, one gets a cut-free sequent calculus for $\msf{S_{BDI}}$.
Let $\mbf{S}_\mathbb{BDI}(\Phi)$ be an extension of $\mbf{S}_\mathbb{BDI}$ with inductive sequents in $\Phi$ as axioms which have algebraic correspondents in $\mc{L}_\bu$. One can transform these axioms into analytic rules, and if these rules which are added to $\msf{G}_\msf{BDFNL}$ does not effect the subformula property, one gets a cut-free sequent system for $\mbf{S}_\mathbb{BDI}(\Phi)$ by omitting additional rules for the two additional operators $\bu$ and $\limp$.

\subsection{The sequent calculus $\msf{G}_\msf{BDFNL}$}

\begin{definition}
Let $\odot$ and $\owedge$ be structural operators for the product $\bu$ and $\wedge$ respectively.
The set of all {\em structures} is defined inductively as follows:
\[
\Gamma::= \phi\mid (\Gamma \odot \Gamma)\mid (\Gamma \owedge \Gamma),
\]
where $\phi\in \mc{L}_\mrm{LC}$. We use $\Gamma, \Delta, \Sigma$ etc. with indexes to denote structures. Each structure $\Gamma$ is associated with a term $\tau(\Gamma) \in \mc{L}_\mrm{LC}$ defined inductively by
\begin{itemize}
\item $\tau(\phi) = \phi$, for every $\phi\in\mc{L}_\mrm{LC}$;
\item $\tau(\Gamma \odot\Delta) = \tau(\Gamma)\bu \tau(\Delta)$;
\item $\tau(\Gamma \owedge \Delta) =\tau(\Gamma)\wedge \tau(\Delta)$.
\end{itemize}
A {\em consecution} (sequent) is $\Gamma \vdash \phi$ where $\Gamma$ is a structure and $\phi$ is an $\mc{L}_\mrm{LC}$-formula.
\end{definition}

Given a BDRG $\mf{A}$ and an assingnment $\mu$ in $A$, for any structure $\Gamma$, define $\mu(\Gamma) = \mu(\tau(\Gamma))$.
We say that $\Gamma\vdash\phi$ is {\em valid} in $\mf{A}$ if $\mu(\Gamma)\leq\mu(\phi)$ for every assignment in $\mf{A}$. We use the notation $\mathbb{BDRG}\models\Gamma\vdash\phi$ to denote that $\Gamma\vdash\phi$ is valid in every BDRG. Obviously $\mathbb{BDRG}\models\Gamma\vdash\phi$ iff $\mathbb{BDRG}\models \tau(\Gamma)\vdash\phi$.

A {\em context} is a structure $\Gamma[-]$ with a single hole $-$ for a structure. Formally, contexts are defined inductively by the following rule:
\[
\Gamma[-] ::= [-] \mid \Gamma[-] \odot \Delta\mid \Delta \odot \Gamma[-]\mid \Gamma[-] \owedge \Delta\mid \Delta \owedge \Gamma[-],
\]
where $\Delta$ is a structure. For any context $\Gamma[-]$ and structure $\Delta$, let $\Gamma[\Delta]$ be the structure obtained from $\Gamma[-]$ by substituting $\Delta$ for the hole $-$. For a context $\Gamma[-]$, let $\tau(\Gamma[-])$ be the formula which contains a hole. In particular, let $\tau([-]) = -$.

\begin{definition}
The sequent calculus $\msf{G_{BDFNL}}$ consists of the following axioms and rules:
\begin{itemize}
\item Axioms:
\[
(\mathrm{Id}) ~ \phi\vdash \phi,
\quad
(\top)~ \Gamma\vdash \top,
\quad
(\bot)~ \Gamma[\bot]\vdash \phi,
\]
\item Logical rules:
\[
\frac{\Delta \vdash \phi \quad \Gamma[\psi] \vdash \gamma}{\Gamma[\Delta \odot (\phi \imp \psi)] \vdash \gamma}(\imp\ \vdash),
\quad
\frac{\phi \odot \Gamma \vdash \psi}{\Gamma \vdash \phi \imp \psi}(\vdash\ \imp),
\]
\[
\frac{\Gamma[\phi] \vdash \gamma\quad \Delta \vdash \psi}{\Gamma[(\phi\leftarrow \psi) \odot \Delta] \vdash \gamma}(\leftarrow\ \vdash),
\quad
\frac{\Gamma \odot\psi \vdash \phi}{\Gamma \vdash \phi\leftarrow \psi}(\vdash\ \leftarrow),
\]
\[
\frac{\Gamma[\phi \odot \psi] \vdash \gamma}{\Gamma[\phi \bu \psi] \vdash \gamma}(\bu\vdash),
~~
\frac{\Gamma \vdash \phi \quad \Delta \vdash \psi}{\Gamma \odot \Delta \vdash \phi \bu \psi}(\vdash\bu),
\]
\[
\frac{\Gamma[\phi\owedge \psi]\vdash \gamma}{\Gamma[\phi\wedge \psi]\vdash \gamma}(\wedge\vdash),
~~
\frac{\Gamma\vdash \phi\quad \Delta \vdash \psi}{\Gamma \owedge \Delta \vdash \phi\wedge \psi}(\vdash\wedge),
\]
\[
\frac{\Gamma[\phi]\vdash \gamma
\quad
\Gamma[\psi]\vdash \gamma}{\Gamma[\phi\vee \psi] \vdash \gamma}(\vee\vdash),
\quad
\frac{\Gamma \vdash \phi_i}{\Gamma\vdash \phi_1\vee \phi_2}(\vdash\vee)(i=1,2),
\]
\item Structural rules:
\[
\frac{\Gamma[\Delta\owedge \Delta]\vdash \phi}{\Gamma [\Delta]\vdash \phi}(\mathrm{\owedge C}),
\quad
\frac{\Gamma[\Delta]\vdash \phi}{\Gamma[\Sigma\owedge\Delta]\vdash \phi}(\mathrm{\owedge W}),
\]
\[
\frac{\Gamma[\Delta\owedge \Lambda]\vdash \phi}{\Gamma[ \Lambda\owedge \Delta]\vdash \phi}(\mathrm{\owedge E}),
\quad
\frac{\Gamma[(\Delta_1\owedge \Delta_2)\owedge \Delta_3] \vdash \phi}{\Gamma[\Delta_1\owedge (\Delta_2\owedge \Delta_3)] \vdash \phi}(\mathrm{\owedge As}).
\]
\end{itemize}
\end{definition}

A derivation in $\msf{G_{BDFNL}}$ is an instance of an axiom or a tree of applications of logical or structural rules. The height of a derivation if the greatest number of successive applications of rules in it, where an axiom has height $0$. A formula with the connective in a logical rule is called the {\em principal} formula of that rule.
A sequent $\Gamma\vdash\phi$ is derivable in $\msf{G_{BDFNL}}$ if there is a derivation ending with $\Gamma\vdash\phi$ in $\msf{G_{BDFNL}}$. A rule of sequents is derivable in $\msf{G_{BDFNL}}$ if the conclusion is derivable whenever the premisses are derivable in $\msf{G_{BDFNL}}$.

\begin{fact}
The following structural rules are derivable in $\msf{G_{BDFNL}}$:
\[
(\mathrm{\owedge W'}) ~\frac{\Gamma[\Delta]\vdash \phi}{\Gamma[\Delta\owedge\Sigma]\vdash \phi},
\quad
(\mathrm{\owedge As'}) ~ \frac{\Gamma[\Delta_1\owedge( \Delta_2\owedge \Delta_3)] \vdash \phi}{\Gamma[(\Delta_1\owedge \Delta_2)\owedge \Delta_3] \vdash \phi}.
\]
\end{fact}

We will now prove the admissibility of cut rule in $\msf{G_{BDFNL}}$. The standard cut rule for a `deep inference' system using contexts with a hole is the following:
\[
\AxiomC{$\Delta\vdash\phi$\quad$\Gamma[\phi]\vdash\psi$}
\RightLabel{\footnotesize $(\mrm{cut})$}
\UnaryInfC{$\Gamma[\Delta]\vdash \psi$}
\DisplayProof
\]
Consider the cut in which the right premiss is obtained by $(\owedge\mrm{C})$ and the left premiss is an axiom $(\top)$:
\[
\AxiomC{$\Delta\vdash\top$}
\AxiomC{$\Gamma[\top\owedge\top]\vdash\psi$}
\RightLabel{\footnotesize $(\owedge\mrm{C})$}
\UnaryInfC{$\Gamma[\top]\vdash\psi$}
\RightLabel{\footnotesize $(\mrm{cut})$}
\BinaryInfC{$\Gamma[\Delta]\vdash \psi$}
\DisplayProof
\]
To eliminate the cut here, one need to cut simultaneously the two occurrences of $\top$ in the premiss of $(\owedge\mrm{C})$. Then we will consider Gentzen's multi-cut or mix rule of which the cut rule is a special case. We use multiple-hole contexts of the form $\Gamma[-]\ldots[-]$ to formulate the mix rule.

\begin{theorem}\label{thm:cut}
The mix rule 
\begin{displaymath}
\frac{\Delta \vdash \phi \quad \Gamma[\phi]\ldots[\phi] \vdash \psi}{\Gamma[\Delta]\ldots[\Delta]\vdash \psi}(\mathrm{mix})
\end{displaymath}
is admissible in $\msf{G_{BDFNL}}$.
\end{theorem}
\begin{proof}
We prove (mix) by simultaneous induction on (i) the complexity of the mixed formula $\phi$; (ii) the height of the derivation of $\Delta\vdash \phi$; (iii) the height of the derivation of $\Gamma[\phi]\vdash \psi$. Assume that $\Delta\vdash \phi$ is obtained by $R_1$, and $\Gamma[\phi]\vdash \psi$ by $R_2$. We have four cases:

\textbf{(I) At least one of $R_1$ and $R_2$ is an axiom}. 
We have two cases:

Case 1. Both $R_1$ and $R_2$ are axioms. We have the following subcases:

(1.1) $R_1 = (\bot)$ or $R_2 = (\top)$. The conclusion of (mix) is an instance of $(\bot)$ or $(\top)$.

(1.2) $R_1 = (\mrm{Id})$. Then $\Delta = \phi$. The conclusion of (mix) is obtained by $R_2$.

(1.3) $R_1 = (\top)$, $R_2 = (\mrm{Id})$. Then $\phi = \top = \psi$. The conclusion of (mix) is obtained by $(\top)$.

(1.4) $R_1 = (\top)$, $R_2 = (\bot)$. Then $\phi = \top$, and $\bot$ occurs in $\Gamma[\top]\ldots[\top]$. The conclusion of (mix) is obtained by $(\bot)$.

Case 2. Exactly one of $R_1$ and $R_2$ is an axiom. We have the following subcases:

(2.1) $R_1 = (\mrm{Id})$. Then the conclusion is the same as the right premiss of (mix).

(2.2) $R_1 = (\bot)$. Then the conclusion of (mix) is an axiom.

(2.3) $R_1 = (\top)$. Then $\phi = \top$. We have subcases according to $R_2$. If $R_2$ is a right rule of a logical connective. We first apply (mix) to $\Delta\vdash\top$ and the premiss(es) of $R_2$, and then apply the rule $R_2$. If $R_2$ is a left rule of a logical connective, the proof is similar to Case 6.
If $R_2$ is a structural rule, the proof is similar to Case 4.

(2.4) $R_2 = (\mrm{Id})$. The conclusion of (mix) is the same as the left premiss of (mix).

(2.5) $R_2 =(\top)$. The conclusion of (mix) is an axiom.

(2.6) $R_2 = (\bot)$. If $\phi\neq\bot$, then the conclusion of (mix) is an instance of $(\bot)$. Suppose $\phi = \bot$. We have subcases according to $R_1$. Clearly $R_1$ cannot be a right rule of a logical connective. If $R_1$ is a left rule of a logical cognitive, the proof is similar to Case 5.
If $R_1$ is a structural rule, the proof is similar to Case 3.

\textbf{(II) At least one of $R_1$ and $R_2$ is a structural rule}. We have two cases:

Case 3. $R_1$ is a structural rule. By induction (ii), the (mix) can be push up to the premiss of $R_1$ and then apply $R_1$. For example, let $R_1=(\owedge\mrm{C})$. The derivation
\[
\AxiomC{$\Delta'[\Sigma\owedge\Sigma] \vdash \phi$}
\RightLabel{\footnotesize $(\owedge\mrm{C})$}
\UnaryInfC{$\Delta'[\Sigma] \vdash \phi$}
\AxiomC{$\Gamma[\phi]\ldots[\phi]\vdash \psi$}
\RightLabel{\footnotesize $(\mrm{mix})$}
\BinaryInfC{$\Gamma[\Delta'[\Sigma]]\ldots\Gamma[\Delta'[\Sigma]]\vdash \psi$}
\DisplayProof
\]
is transformed into
\[
\AxiomC{$\Delta'[\Sigma\owedge\Sigma] \vdash \phi$\quad $\Gamma[\phi]\ldots[\phi]\vdash \psi$}
\RightLabel{\footnotesize $(\mrm{mix})$}
\UnaryInfC{$\Gamma[\Delta'[\Sigma\owedge\Sigma]]\ldots[\Delta'[\Sigma\owedge\Sigma]]\vdash \psi$}
\RightLabel{\footnotesize $(\owedge\mrm{C}^*)$}
\UnaryInfC{$\Gamma[\Delta'[\Sigma]]\ldots\Gamma[\Delta'[\Sigma]]\vdash \psi$}
\DisplayProof
\]
where $(\owedge\mrm{C}^*)$ stands for the application of $(\owedge\mrm{C})$ multiple times.

Case 4. $R_2$ is a structural rule. Suppose that $\phi$ is obtained by $(\owedge\mrm{W})$ in $R_2$. The derivation
\[
\AxiomC{$\Delta \vdash \phi$}
\AxiomC{$\Gamma[\phi]\ldots[\Delta']\ldots[\phi]\vdash \psi$}
\RightLabel{\footnotesize $(\owedge\mrm{W})$}
\UnaryInfC{$\Gamma[\phi]\ldots[\Sigma[\phi]\owedge\Delta']\ldots[\phi]\vdash \psi$}
\RightLabel{\footnotesize $(\mrm{mix})$}
\BinaryInfC{$\Gamma[\Delta]\ldots[\Sigma[\Delta]\owedge\Delta']\ldots[\Delta]\vdash \psi$}
\DisplayProof
\]
is transformed into
\[
\AxiomC{$\Delta \vdash \phi$\quad$\Gamma[\phi]\ldots[\Delta']\ldots[\phi]\vdash \psi$}
\RightLabel{\footnotesize $(\mrm{mix})$}
\UnaryInfC{$\Gamma[\Delta]\ldots[\Delta']\ldots[\Delta]\vdash \psi$}
\RightLabel{\footnotesize $(\owedge\mrm{W})$}
\UnaryInfC{$\Gamma[\Delta]\ldots[\Sigma[\Delta]\owedge\Delta']\ldots[\Delta]\vdash \psi$}
\DisplayProof
\]
For the remaining cases of $R_2$, by induction (ii), the (mix) can be push up to the premiss of $R_2$ and then apply $R_2$.

\textbf{(III) At least one of $R_1$ and $R_2$ is a logical rule, but the mixed formula is not principal}. We have two cases:

Case 5. The mixed formula $\phi$ is not principal in the left premiss. Then we have subcases according to $R_1$. Clearly $R_1$ cannot be a right rule of a logical connective. Assume $R_1 = (\imp\ \vdash)$. The derivation ends with
\[
\AxiomC{$\Delta'\vdash\chi$\quad$\Delta[\delta]\vdash\phi$}
\RightLabel{\footnotesize $(\imp\ \vdash)$}
\UnaryInfC{$\Delta[\Delta'\odot(\chi\imp\delta)]\vdash\phi$}
\AxiomC{$\Gamma[\phi]\ldots[\phi]\vdash\psi$}
\RightLabel{\footnotesize $(\mrm{mix})$}
\BinaryInfC{$\Gamma[\Delta[\Delta'\odot(\chi\imp\delta)]]\ldots[\Delta[\Delta'\odot(\chi\imp\delta)]]\vdash\psi$}
\DisplayProof
\]
Firstly we push up (mix) as below:
\[
\AxiomC{$\Delta[\delta]\vdash\phi$\quad$\Gamma[\phi]\ldots[\phi]\vdash\psi$}
\RightLabel{\footnotesize $(\mrm{mix})$}
\UnaryInfC{$\Gamma[\Delta[\delta]]\ldots[\Delta[\delta]]\vdash\psi$}
\DisplayProof
\]
Then we apply $(\imp\ \vdash)$ to $\Delta'\vdash\chi$ and 
$\Gamma[\Delta[\delta]]\ldots[\Delta[\delta]]\vdash\psi$ multiple times, and we get the conclusion. 

Assume $R_1 = (\vee\vdash)$. The derivation ends with
\[
\AxiomC{$\Delta[\chi]\vdash\phi$\quad$\Delta[\delta]\vdash\phi$}
\RightLabel{\footnotesize $(\vee\vdash)$}
\UnaryInfC{$\Delta[\chi\vee\delta]\vdash\phi$}
\AxiomC{$\Gamma[\phi]\ldots[\phi]\vdash\psi$}
\RightLabel{\footnotesize $(\mrm{mix})$}
\BinaryInfC{$\Gamma[\Delta[\chi\vee\delta]]\ldots[\Delta[\chi\vee\delta]]\vdash\psi$}
\DisplayProof
\]
The rule (mix) is push up to sequents with less height of derivation in multiple steps. For the first occurrence of $\phi$ in $\Gamma[\phi]\ldots[\phi]\vdash\psi$, mix it with $\Delta[\chi]\vdash\phi$ and $\Delta[\delta]\vdash\phi$ respectively, and by $(\vee\vdash)$ one gets $\Gamma[\Delta[\chi\vee\delta]][\phi]\ldots[\phi]\vdash\psi$. Repeat this process multiple times and we achieve the conclusion $\Gamma[\Delta[\chi\vee\delta]]\ldots[\Delta[\chi\vee\delta]]\vdash\psi$.

The remaining cases $R_1 = (\limp\ \vdash)$, $(\bu\vdash)$ or $(\wedge\vdash)$ are similar.

Case 6. The mixed formula $\phi$ is principal only in the left premiss. Then we have subcases according to $R_2$. Assume $R_2 = (\imp\ \vdash)$. If $\phi$ does not occur in $\Sigma$, then the derivation
\[
\AxiomC{$\Delta \vdash \phi$}
\AxiomC{$\Sigma\vdash\chi$\quad$\Gamma'[\delta][\phi]\ldots[\phi]\vdash\psi$}
\RightLabel{\footnotesize $(\imp\ \vdash)$}
\UnaryInfC{$\Gamma'[\Sigma\odot(\chi\imp\delta)][\phi]\ldots[\phi]\vdash \psi$}
\RightLabel{\footnotesize $(\mrm{mix})$}
\BinaryInfC{$\Gamma'[\Sigma\odot(\chi\imp\delta)][\Delta]\ldots[\Delta]\vdash \psi$}
\DisplayProof
\]
is transformed into
\[
\AxiomC{$\Sigma\vdash\chi$}
\AxiomC{$\Delta \vdash \phi$\quad$\Gamma'[\delta][\phi]\ldots[\phi]\vdash\psi$}
\RightLabel{\footnotesize $(\mrm{mix})$}
\UnaryInfC{$\Gamma'[\delta ][\Delta]\ldots[\Delta]\vdash \psi$}
\RightLabel{\footnotesize $(\imp\ \vdash)$}
\BinaryInfC{$\Gamma'[\Sigma\odot(\chi\imp\delta)][\Delta]\ldots[\Delta]\vdash \psi$}
\DisplayProof
\]
Suppose that $\Sigma = \Sigma'[\phi]$. The derivation
 \[
\AxiomC{$\Delta \vdash \phi$}
\AxiomC{$\Sigma'[\phi]\vdash\chi$\quad$\Gamma'[\delta][\phi]\ldots[\phi]\vdash\psi$}
\RightLabel{\footnotesize $(\imp\ \vdash)$}
\UnaryInfC{$\Gamma'[\Sigma'[\phi]\odot(\chi\imp\delta)][\phi]\ldots[\phi]\vdash \psi$}
\RightLabel{\footnotesize $(\mrm{mix})$}
\BinaryInfC{$\Gamma'[\Sigma'[\Delta]\odot(\chi\imp\delta)][\Delta]\ldots[\Delta]\vdash \psi$}
\DisplayProof
\]
is transformed into
\[
\AxiomC{$\Delta\vdash\phi$\quad$\Sigma'[\phi]\vdash\chi$}
\RightLabel{\footnotesize $(\mrm{mix})$}
\UnaryInfC{$\Sigma'[\Delta]\vdash\chi$}
\AxiomC{$\Delta \vdash \phi$\quad$\Gamma'[\delta][\phi]\ldots[\phi]\vdash\psi$}
\RightLabel{\footnotesize $(\mrm{mix})$}
\UnaryInfC{$\Gamma'[\delta ][\Delta]\ldots[\Delta]\vdash \psi$}
\RightLabel{\footnotesize $(\imp\ \vdash)$}
\BinaryInfC{$\Gamma'[\Sigma'[\Delta]\odot(\chi\imp\delta)][\Delta]\ldots[\Delta]\vdash \psi$}\DisplayProof
\]
The remaining cases $R_2 = (\limp\ \vdash)$, $(\bu\vdash)$, $(\wedge\vdash)$, or $(\vee\vdash)$ are similar.

\textbf{(IV) Both $R_1$ and $R_2$ are logical rules, and the mixed formula is principal}. Then we prove it by induction on the complexity of $\phi$.
Assume that $\phi=\phi_1\bu \phi_2$.  The derivation
\[
\AxiomC{$\Delta_1\vdash \phi_1$\quad $\Delta_2\vdash \phi_2$}
\RightLabel{\footnotesize $(\vdash\bu)$}
\UnaryInfC{$\Delta_1\odot\Delta_2\vdash \phi$}
\AxiomC{$\Gamma[\phi]\ldots[\phi_1\odot \phi_2]\ldots[\phi]\vdash \psi$}
\RightLabel{\footnotesize $(\bu\vdash)$}
\UnaryInfC{$\Gamma[\phi]\ldots[\phi_1\bu \phi_2]\ldots[\phi]\vdash \psi$}
\RightLabel{\footnotesize $(\mrm{mix})$}
\BinaryInfC{$\Gamma[\Delta_1\odot \Delta_2]\vdash \psi$}
\DisplayProof
\]
is transformed into
\[
\AxiomC{$\Delta_1\vdash\phi_1$}
\AxiomC{$\Delta_2\vdash\phi_2$}
\AxiomC{$\Delta_1\odot\Delta_2\vdash \phi$\quad$\Gamma[\phi]\ldots[\phi_1\odot \phi_2]\ldots[\phi]\vdash \psi$}
\RightLabel{\footnotesize $(\mrm{mix})$}
\UnaryInfC{$\Gamma[\Delta_1\odot \Delta_2]\ldots[\phi_1\odot \phi_2]\ldots[\Delta_1\odot \Delta_2]\vdash \psi$}
\RightLabel{\footnotesize $(\mrm{mix})$}
\BinaryInfC{$\Gamma[\Delta_1\odot \Delta_2]\ldots[\phi_1\odot \Delta_2]\ldots[\Delta_1\odot \Delta_2]\vdash \psi$}
\RightLabel{\footnotesize $(\mrm{mix})$}
\BinaryInfC{$\Gamma[\Delta_1\odot \Delta_2]\ldots[\Delta_1\odot \Delta_2]\ldots[\Delta_1\odot \Delta_2]\vdash \psi$}
\DisplayProof
\]
Note that the (mix) rule is push up to sequents with lesser height in the derivation.
The remaining cases $\phi=\phi_1\imp \phi_2$, $\phi_1\leftarrow \phi_2$, $\phi_1\wedge \phi_2$, or $\phi_1\vee \phi_2$ are quite similar.
\qed
\end{proof}

In all rules of $\msf{G_{BDFNL}}$, no formula disappears in from the premiss(es) to the conclusion. Hence we get the subformula property of $\msf{G_{BDFNL}}$ immediately:

\begin{theorem}
If a consecution $\Gamma\vdash\phi$ has a derivation in $\msf{G_{BDFNL}}$, then all formulas in the derivation are subformulas of $\Gamma,\phi$.
\end{theorem}

Now we will prove the completeness of $\msf{G_{BDFNL}}$ with respect to $\mathbb{BDRG}$. Firstly, we have the following lemma on the invertibility of some rules in $\msf{G_{BDFNL}}$:

\begin{lemma}\label{lemma:inv}
The following rules are admissible in $\msf{G_{BDFNL}}$:
\[
\frac{\Gamma\vdash\phi\imp\psi}{\phi\odot\Gamma\vdash\psi}(\vdash\imp\ \ua),
\quad
\frac{\Gamma[\phi\bu\psi]\ldots[\phi\bu\psi]\vdash\gamma}{\Gamma[\phi\odot\psi]\ldots[\phi\odot\psi]\vdash\gamma}(\bu\vdash\ \ua),
\]
\[
\frac{\Gamma \vdash \phi\leftarrow \psi}{\Gamma \odot\psi \vdash \phi}(\vdash\leftarrow\ \ua),
\quad
\frac{\Gamma[\phi\wedge\psi]\ldots[\phi\wedge\psi]\vdash\gamma}{\Gamma[\phi\owedge\psi]\ldots[\phi\owedge\psi]\vdash\gamma}(\owedge\vdash\ \ua).
\]
\end{lemma}
\begin{proof}
The proof is done by induction on the height of the derivation of the premiss. Here we prove only the admissibility of $(\vdash\imp\ \ua)$ and $(\bu\vdash\ \ua)$. The remaining rules are shown similarly. Assume that the premiss is obtained by $R$.

For $(\vdash\imp\ \ua)$, if $R$ is an axiom, one can get $\phi\odot\Gamma\vdash$ easily. If $R$ is a left rule of a connective, or a rule for $\owedge$,  we push up $(\vdash\imp\ \ua)$ to the premiss(es) of $R$ and then apply the rule $R$. If $R$ is a right rule, it can only be $(\vdash\imp)$ and then one gets $\phi\odot\Gamma\vdash\psi$.

For $(\bu\vdash\ \ua)$, assume that $\Gamma\vdash\phi\imp\psi$ is obtained by $R$. We have the following cases:

Case 1. $R$ is an axiom. When $R$ is $(\bot)$ or $(\top)$, the conclusion is also $(\bot)$ or $(\top)$. Assume $R = (\mrm{Id})$.
The conclusion $\phi\odot\psi\vdash\phi\bu\psi$ can be derived by $(\vdash\bu)$ obviously.

Case 2. $R$ is a logical rule. If $R$ is a rule of $\imp$, $\limp$, $\wedge$, or $R$ is $(\vdash\bu)$, one can push up $(\vdash\imp\ \ua)$ to the premiss of $R$ and then apply the rule $R$. If $R=(\bu\vdash)$, one can push up $(\bu\vdash\ \ua)$ to the premiss of $R$ and obtain the conclusion directly.

Case 3. $R$ is a structural rule. Apply $(\bu\vdash\ \ua)$ to the premiss of $R$ and then apply $R$.
\qed
\end{proof}

\begin{lemma}\label{lemma:gen1}
If $\phi\vdash\psi$ is derivable in $\msf{BDFNL}$, then $\phi\vdash\psi$ is derivable in $\msf{G_{BDFNL}}$.
\end{lemma}
\begin{proof}
By induction on the derivation of $\phi\vdash\psi$ in $\msf{BDFNL}$. 

Case 1. $\phi\vdash\psi$ is an axiom. The cases of (Id), ($\top$) and ($\bot$) are clear. For (D), one derivation is
\[
\AxiomC{$\phi\vdash\phi\quad\psi\vdash\psi$}
\RightLabel{\footnotesize $\mrm{(\vdash\wedge)}$}
\UnaryInfC{$\phi\owedge\psi\Imp \phi\wedge\psi$}
\RightLabel{\footnotesize $\mrm{(\vdash\vee)}$}
\UnaryInfC{$\phi\owedge\psi\Imp (\phi\wedge\psi)\vee(\phi\wedge\gamma)$}
\AxiomC{$\phi\vdash\phi\quad\gamma\vdash\gamma$}
\RightLabel{\footnotesize $\mrm{(\vdash\wedge)}$}
\UnaryInfC{$\phi\owedge\gamma\Imp \phi\wedge\gamma$}
\RightLabel{\footnotesize $\mrm{(\vdash\vee)}$}
\UnaryInfC{$\phi\owedge\gamma\Imp (\phi\wedge\psi)\vee(\phi\wedge\gamma)$}
\RightLabel{\footnotesize $\mrm{(\vee\vdash)}$}
\BinaryInfC{$\phi\owedge(\psi\vee\gamma)\Imp (\phi\wedge\psi)\vee(\phi\wedge\gamma)$}
\RightLabel{\footnotesize $\mrm{(\wedge\vdash)}$}
\UnaryInfC{$\phi\wedge(\psi\vee\gamma)\Imp (\phi\wedge\psi)\vee(\phi\wedge\gamma)$}
\DisplayProof
\]

Case 2. $\phi\vdash\psi$ is obtained by a rule. Obviously, rules for $\wedge$ and $\vee$ are derivable in $\msf{G_\msf{BDFNL}}$. The rule (cut) is a special case of (mix) in $\msf{G_\msf{BDFNL}}$.
For residuation rules, (Res1) is shown by the rule $(\bu\vdash\ \ua)$ in Lemma \ref{lemma:inv} and $(\bu\vdash)$. (Res2) is obtained by the rule $(\vdash\imp\ \ua)$ in Lemma \ref{lemma:inv} and $(\bu\vdash)$. The remaining residuation rules are shown similarly.
\qed
\end{proof}

\begin{lemma}\label{lemma:gen2}
If a consecution $\Gamma\vdash\phi$ is derivable in $\msf{G_\msf{BDFNL}}$, then $\tau(\Gamma)\vdash\phi$ is derivable in $\msf{BDFNL}$.
\end{lemma}
\begin{proof}
By induction on the height of the derivation of $\Gamma\vdash\phi$ in $\msf{G_\msf{BDFNL}}$.

Case 1. $\Gamma\vdash\phi$ is an axiom. The cases of (Id) and ($\top$) are obvious. 
We prove $\tau(\Gamma[\bot])\vdash\phi$ by induction on the construction of $\Gamma$. We have the following cases:

(1.1) $\Gamma = \psi$. Then $\Gamma[\bot] = \bot = \psi$. By $(\bot)$, we have $\bot\vdash\phi$. 

(1.2) $\Gamma = \Gamma'\odot\Delta$. Assume $\Gamma' = \Gamma'[\bot]$. By induction hypothesis, we have $\tau(\Gamma'[\bot])\vdash\phi\limp \tau(\Delta)$. Then by (Res4) in $\msf{BDFNL}$, one gets $\tau(\Gamma'[\bot])\bu \tau(\Delta)\vdash\phi$. Assume $\Delta=\Delta[\bot]$. By induction hypothesis, $\tau(\Delta[\bot])\vdash \tau(\Gamma)\imp\phi$. By (Res2), one gets $\tau(\Gamma)\bu \tau(\Delta[\bot])\vdash\phi$.

(1.3) $\Gamma = \Gamma'\owedge\Delta$.
Then $\tau(\Gamma[\bot]) = \tau(\Gamma'[\bot])\wedge \tau(\Delta)$ or $\tau(\Gamma[\bot]) = \tau(\Gamma')\wedge \tau(\Delta[\bot])$. By induction hypothesis, one can easily obtain $\tau(\Gamma)\vdash\phi$.

Case 2. $\Gamma\vdash\phi$ is obtained by $(\imp\ \vdash)$ or $(\limp\ \vdash)$. We prove the case of $(\imp\ \vdash)$ and the other one is similar. By inductive hypothesis, we have $\tau(\Delta)\vdash\chi$ and $\tau(\Sigma[\xi])\vdash\phi$. Our goal is to prove $\tau(\Sigma[\Delta\odot(\chi\imp\xi)])\vdash\phi$. Firstly, in $\msf{BDFNL}$,
from $\tau(\Delta)\vdash\chi$, one gets $\chi\imp\xi\vdash \tau(\Delta)\imp\xi$. Then by (Res2), one gets
$\tau(\Delta)\bu(\chi\imp\xi)\vdash\xi$. 

\textbf{Claim}. For any context $\Sigma[-]$, we have $\tau(\Sigma[\tau(\Delta)\bu(\chi\imp\xi)])\vdash \tau(\Sigma[\xi])$.

{\em Proof of Claim}. By induction on the construction of $\Sigma[-]$. The case $\Sigma[-] = [-]$ is obvious. Assume $\Sigma[-] = \Sigma'[-]\odot\Delta'$. Then $\tau(\Sigma[-]) = \tau(\Sigma'[-])\bu \tau(\Delta')$. By induction hypothesis, one gets $\tau(\Sigma'[\tau(\Delta)\bu(\chi\imp\xi)])\vdash \tau(\Sigma'[\xi])$. Then one gets
$\tau(\Sigma'[\tau(\Delta)\bu(\chi\imp\xi)])\bu \tau(\Delta')\vdash \tau(\Sigma'[\xi])\bu \tau(\Delta')$.
The remaining cases are similar. This completes the proof of the claim.

Now by applying (cut) to $\tau(\Sigma[\tau(\Delta)\bu(\chi\imp\xi)])\vdash \tau(\Sigma[\xi])$ and $\tau(\Sigma[\xi])\vdash\phi$, one gets $\tau(\Sigma[\Delta\odot(\chi\imp\xi)]\vdash\phi)$.

Case 3. $\Gamma\vdash\phi$ is obtained by $(\vdash\ \imp)$ or $(\vdash\ \limp)$. We prove the case of $(\vdash\ \imp)$ and the other one is similar. Let $\phi = \chi\imp\xi$. From the premiss $\chi\odot \Gamma\vdash\xi$ of $(\vdash\ \imp)$, by inductive hypothesis, one gets $\chi\bu \tau(\Gamma)\vdash\xi$. By (Res1), one gets $\tau(\Gamma)\vdash \chi\imp\xi$. 

Case 4. $\Gamma\vdash\phi$ is obtained by $(\bu\vdash)$.
By induction hypothesis, one gets $\tau(\Gamma[\chi\odot\xi])\vdash\phi$ is derivable in $\msf{BDFNL}$. Clearly, it is rather easy to check by induction on the construction of $\Gamma$ that $\tau(\Gamma[\chi\odot\xi]) = \tau(\Gamma[\chi\bu\xi])$. 
Therefore $\tau(\Gamma[\chi\bu\xi])\vdash\phi$ is derivable in $\msf{BDFNL}$.

Case 5. $\Gamma\vdash\phi$ is obtained by $(\vdash\bu)$. Let $\phi = \chi\bu\xi$. By induction hypothesis, one gets $\tau(\Gamma)\vdash\chi$ and $\tau(\Delta)\vdash\xi$. By the monotonicity rules of $\bu$, one gets $\tau(\Gamma)\bu\tau(\Delta)\vdash\chi\bu\xi$. 

Case 6. $\Gamma\vdash\phi$ is obtained by $(\wedge\vdash)$ or $(\vdash\wedge)$. The proof is similar to Case 4 or Case 5.

Case 7. $\Gamma\vdash\phi$ is obtained by $(\vee\vdash)$. By induction hypothesis, one gets $\tau(\Gamma[\chi])\vdash\phi$ and $\tau(\Gamma[\xi])\vdash\phi$. We prove $\tau(\Gamma[\chi\vee\xi])\vdash\phi$ by induction on the construction of $\Gamma$. 

(7.1) $\Gamma[-] = [-]$. Then we have $\chi\vdash\phi$ and $\xi\vdash\phi$. By $(\vee\mrm{L})$ in $\msf{BDFNL}$, one gets $\chi\vee\xi\vdash\phi$.

(7.2) $\Gamma[-] = \Gamma_1[-]\odot\Gamma_2$ or $\Gamma_1\odot\Gamma_2[-]$. The two cases are quite similar, and we specify only the first case. Clearly we have $\tau(\Gamma_1[\chi])\bu\tau(\Gamma_2)\vdash\phi$ and $\tau(\Gamma_1[\xi])\bu\tau(\Gamma_2)\vdash\phi$. By residuation rules, one gets $\tau(\Gamma_1[\chi])\vdash\phi\limp \tau(\Gamma_2)$ and $\tau(\Gamma_1[\xi]) \vdash\phi\limp \tau(\Gamma_2)$. By induction hypothesis on $\Gamma_1$, one gets $\tau(\Gamma_1[\chi\vee\xi])\vdash\phi\limp\tau(\Gamma_2)$. By residuation, one gets $\tau(\Gamma_1[\chi\vee\xi])\bu\tau(\Gamma_2)\vdash\phi$.

(7.3) $\Gamma[-] = \Gamma_1[-]\owedge\Gamma_2$ or $\Gamma_1\owedge\Gamma_2[-]$. The proof is quite similar to (7.2).

Case 8. $\Gamma\vdash\phi$ is obtained by $(\vdash\vee)$. The proof is quite similar to Case 5.

Case 9. $\Gamma\vdash\phi$ is obtained by $(\owedge\mrm{W})$. By induction hypothesis, one gets $\tau(\Gamma[\Delta])\vdash\phi$. Clearly one gets $\tau(\Gamma[\tau(\Delta)]\vdash\phi$. We prove $\tau(\Gamma[\tau(\Sigma)\wedge\tau(\Delta)])\vdash\phi$ by induction on $\Gamma$.

(9.1) $\Gamma[-] = [-]$. Then we have $\tau(\Delta)\vdash\phi$. In $\msf{BDFNL}$ we have $\tau(\Sigma)\wedge\tau(\Delta)\vdash\phi$.

(9.2) $\Gamma[-] = \Gamma_1[-]\odot\Gamma_2$ or $\Gamma_1\odot\Gamma_2[-]$. The two cases are quite similar, and we specify only the first case. Clearly $\tau(\Gamma_1[\tau(\Delta)])\bu\tau(\Gamma_2)\vdash\phi$. By residuation, one gets $\tau(\Gamma_1[\tau(\Delta)])\vdash\phi\limp \tau(\Gamma_2)$. By induction hypothesis on $\Gamma_1$, one gets $\tau(\Gamma_1[\tau(\Sigma)\wedge\tau(\Delta)])\vdash\phi\limp\tau(\Gamma_2)$. By residuation, one gets $\tau(\Gamma_1[\tau(\Sigma)\wedge\tau(\Delta)])\bu\tau(\Gamma_2)\vdash\phi$.

(9.3) $\Gamma[-] = \Gamma_1[-]\owedge\Gamma_2$ or $\Gamma_1\owedge\Gamma_2[\chi]$. The proof is quite similar to (9.2).

Case 10. $\Gamma\vdash\phi$ is obtained by $(\owedge\mrm{C})$, $(\owedge\mrm{E})$ or $(\owedge\mrm{As})$. The proof is done by lattice rules in $\msf{BDFNL}$. The proof is quite similar to Case 9.
\qed
\end{proof}

\begin{lemma}\label{lemma:gen3}
If $\tau(\Gamma)\vdash\phi$ is derivable in $\msf{G_\msf{BDFNL}}$, then $\Gamma\vdash \phi$ is derivable in $\msf{G_\msf{BDFNL}}$.\end{lemma}
\begin{proof}
By induction on the construction of $\Gamma$. The case that $\Gamma$ is a formula is obvious. Assume $\Gamma = \Gamma_1\owedge\Gamma_2$. Assume $\tau(\Gamma_1)\wedge\tau(\Gamma_2)\vdash\phi$. By induction on the construction of a structure $\Sigma$ one can easily show $\Sigma\vdash\tau(\Sigma)$. Then we have $\Gamma_1\vdash\tau(\Gamma_1)$ and $\Gamma_2\vdash\tau(\Gamma_2)$. By $(\vdash\wedge)$, one gets $\Gamma_1\owedge\Gamma_2\vdash\tau(\Gamma_1)\wedge\tau(\Gamma_2)$. By (mix), one gets $\Gamma_1\owedge\Gamma_2\vdash\phi$.
The case $\Gamma = \Gamma_1\odot\Gamma_2$ is similar.
\qed
\end{proof}

\begin{theorem}\label{thm:completeness}
A consecution $\Gamma\vdash\phi$ is derivable in $\msf{G_\msf{BDFNL}}$ if and only if $\mathbb{BDRG}\models\Gamma\vdash\phi$.
\end{theorem}
\begin{proof}
For the `if' part, assume $\mathbb{BDRG}\models\Gamma\vdash\phi$. Then $\mathbb{BDRG}\models \tau(\Gamma)\vdash\phi$. By the completeness of $\msf{BDFNL}$, $\tau(\Gamma)\vdash\phi$ is derivable in $\msf{BDFNL}$. By Lemma \ref{lemma:gen1}, 
$\tau(\Gamma)\vdash\phi$ is derivable in $\msf{G_{BDFNL}}$.
By Lemma \ref{lemma:gen3}, $\Gamma\vdash \phi$ is derivable in $\msf{G_{BDFNL}}$. For the `only if' part, assume that $\Gamma\vdash \phi$ is derivable in $\msf{G_{BDFNL}}$. By Lemma \ref{lemma:gen2}, $\tau(\Gamma)\vdash \phi$ is derivable in $\msf{BDFNL}$. By the completeness of $\msf{BDFNL}$, $\mathbb{BDRG}\models  \tau(\Gamma)\vdash\phi$. Therefore $\mathbb{BDRG}\models\Gamma\vdash\phi$.
\qed
\end{proof}

\subsection{Extensions}
We will now consider some extensions of $\msf{S}_\mathbb{BDI}$ and their conservative extensions over $\msf{BDFNL}$.
Given an $\mc{L}_\bu$-sequent $(\sigma)~\chi\vdash\delta$ the propositional variables occurred in which are among $p_1,\ldots, p_n$, the structural rule corresponding to $(\sigma)$ is defined as the following rule $(\odot\sigma)$:
\[
\frac{\delta[\Delta_1/p_1,\ldots,\Delta_n/p_n]\vdash\phi}{\chi[\Delta_1/p_1,\ldots,\Delta_n/p_n]\vdash\phi}{(\odot\sigma)}
\]
where $\delta[\Delta_1/p_1,\ldots,\Delta_n/p_n]$ and $\chi[\Delta_1/p_1,\ldots,\Delta_n/p_n]$ are obtained from $\delta$ and $\chi$ by substituting $\Delta_i$ for $p_i$ uniformly, and substituting $\odot$ for $\bu$.

\begin{example}
For weak Heyting algebras, we have the following structural rules for $(tr)$ and $(wl)$:
\[
\frac{\Gamma[(\Lambda\odot\Delta)\odot \Delta]\vdash \phi}{\Gamma[\Lambda\odot\Delta]\vdash \phi}(\odot tr) ,
\quad
\frac{\Gamma[\Delta]\vdash \phi}{\Gamma[\Delta\odot\Sigma]\vdash \phi}(\odot wl).
\]
Let $\msf{G}_\mathsf{RWH}$ be the Gentzen-style sequent system obtained from $\msf{G}_\msf{RBDI}$ by adding $(\odot tr)$ and $(\odot wl)$. We can get similar sequent rules for sequents in Example \ref{exam:correspondents} and Genzten-style sequent systems.
\end{example}

For any set of $\mc{L}_\bu$-sequents $\Psi$, let $\odot\Psi = \{\odot\sigma\mid \sigma\in\Psi\}$ and $\msf{G_\msf{BDFNL}}(\odot\Psi)$ be the Gentzen-style sequent system obtained from $\msf{G_\msf{BDFNL}}$ by adding all rules in $(\odot\Psi)$.

\begin{theorem}\label{thm:mix-admissible}
For any set of $\mc{L}_\bu$-sequents $\Psi$,
if for every sequent $\chi\vdash\delta\in\Psi$, each propositional variable in $\chi$ occurs only once,
then $\mrm{(mix)}$ is admissible in $\msf{G_\msf{BDFNL}}(\odot\Psi)$.
\end{theorem}
\begin{proof}
Based on the proof of Theorem \ref{thm:cut}, one needs to consider only the case that the right premise of (mix) is obtained by $(\odot\sigma)$. We first apply (mix) to the left premiss of (mix) and the premiss of $(\odot\sigma)$. Then by $(\odot\sigma)$, we get the conclusion of (mix).
\qed
\end{proof}

\begin{remark}
The condition that a propositional variable occurs at most once in $\chi$ in Theorem \ref{thm:mix-admissible} is significant. All sequents in Example \ref{exam:correspondents} satisfy this condition. When a propositional variable occurs more than once in $\chi$, the proof strategy in Theorem  \ref{thm:mix-admissible} may not work. For example, consider the the following inverse rule of $(\odot tr)$ which is obtained from $(p\bu q)\bu q\vdash p\bu q$:
\begin{displaymath}
\frac{\Gamma[\Lambda\odot\Delta[\psi]]\vdash\phi}{\Gamma[(\Lambda\odot\Delta[\psi])\odot\Delta[\psi]]\vdash\phi}(\odot tr\hspace{-0.1em}\ua)
\end{displaymath}
and the derivation
\[
\AxiomC{$\Sigma\vdash\psi$}
\AxiomC{$\Gamma[\Lambda\odot\Delta[\psi]]\vdash\phi$}
\RightLabel{\footnotesize $(\odot tr\hspace{-0.1em}\ua)$}
\UnaryInfC{$\Gamma[(\Lambda\odot\Delta[\psi])\odot\Delta[\psi]]\vdash\phi$}
\RightLabel{\footnotesize $(\mrm{mix})$}
\BinaryInfC{$\Gamma[(\Lambda\odot\Delta[\Sigma])\odot\Delta[\psi]]\vdash\phi$}
\DisplayProof
\]
in which only one occurrence of $\psi$ is mixed. In such a case, we may not be able to push up (mix) to the premiss of $(\odot tr\hspace{-0.1em}\ua)$.
\end{remark}

An $\mc{L}_\bu$-sequent $\chi\vdash\delta$ is said to be 
{\em good} if each propositional variable occurs at most once in $\chi$. Then we have the following theorem about good sequents:

\begin{theorem}
For any set of good $\mc{L}_\bu$-sequents $\Psi$, the following hold:

$(1)$ $\Gamma\vdash\phi$ is derivable in $\msf{G_\msf{BDFNL}}(\odot\Psi)$ iff $\msf{Alg}^+(\Psi)\models\Gamma\vdash \phi$.

$(2)$ if every propositional variable occurred in $\delta$ also occurs in $\chi$ for each sequent $\chi\vdash\delta$ in $\Psi$, then $\msf{G_\msf{BDFNL}}(\odot\Psi)$ has the subformula property.
\end{theorem}
\begin{proof}
The proof of (1) is similar to Theorem \ref{thm:completeness}. It suffices to show that the algebraic sequent system $\msf{BDFNL}(\Psi)$ is equivalent to $\msf{G_{BDFNL}}(\odot\Psi)$. 
For (2), if every propositional variable occurred in $\delta$ also occurs in $\chi$, then every subformula of $\delta$ is a subformula of $\chi$. Hence the structural rule $(\odot\sigma)$ does not effect on the subformula property.
\qed
\end{proof}

Let $\Phi$ be a set of inductive $\mc{L}_\mrm{SI}$-sequents. Assume that $\Psi$ is set of $\mc{L}_\bu$-sequent such that $\Phi\equiv_\msf{ALC}\Psi$. Then the algebraic sequent system $\msf{BDFNL}(\Psi)$ is a conservative extension of $\msf{S}_\mathbb{BDI}(\Phi)$. If $\Psi$ is a set of good $\mc{L}_\bu$-sequent, one gets a Gentzen-style cut-free sequent calculus $\msf{G_{BDFNL}}(\odot\Psi)$.
Furthermore, if $\msf{G_{BDFNL}}(\odot\Psi)$ has the subformula property, we obtain a Gentzen-style cut-free sequent calculus for $\msf{S}_\mathbb{BDI}(\Phi)$ if we omit rules for $\bu$ and $\limp$ from $\msf{G_{BDFNL}}(\odot\Psi)$.

\begin{table}[htbp]
\caption{Gentzen-style Sequent Calculi}
\vspace{-1em}
\[
\begin{tabular}{|c|l|}
\hline
Strict Implication Logic~&~Conservative Extension\\
\hline
$\mathsf{G_{WH}}$~&~$\mathsf{G_{RWH}}=\mathsf{G_{BDFNL}}+(\odot wl)+(\odot tr)$\\
$\mathsf{G_{T}}$~&~$\mathsf{G_{RT}}=\mathsf{G_{BDFNL}}+(\odot ct)$\\
$\mathsf{G_{W}}$~&~$\mathsf{G_{RW}}=\mathsf{G_{BDFNL}}+(\odot wr)$\\
$\mathsf{G_{RT}}$~&~$\mathsf{G_{RRT}}=\mathsf{G_{BDFNL}}+(\odot rt)$\\
$\mathsf{G_{B}}$~&~$\mathsf{G_{RB}}=\mathsf{G_{BDFNL}}+(\odot b)$\\
$\mathsf{G_{B'}}$~&~$\mathsf{G_{RB'}}=\mathsf{G_{BDFNL}}+(\odot b')$\\
$\mathsf{G_{C}}$~&~$\mathsf{G_{RC}}=\mathsf{G_{BDFNL}}+(\odot c)$\\
$\mathsf{G_{FR}}$~&~$\mathsf{G_{RFR}}=\mathsf{G_{BDFNL}}+(\odot fr)$\\
$\mathsf{G_{W'}}$~&~$\mathsf{G_{RW'}}=\mathsf{G_{BDFNL}}+(\odot w')$\\
$\mathsf{G_{BCA}}$~&~$\mathsf{G_{RBCA}}=\mathsf{G_{T}}+(\odot w)$\\
$\mathsf{G_{KT}}$~&~$\mathsf{G_{RKT}}=\mathsf{G_{RWH}}+(\odot ct)$\\
$\mathsf{G_{K4}}$~&~$\mathsf{G_{RK4}}=\mathsf{G_{RWH}}+(\odot rt)$\\
$\mathsf{G_{S4}}$~&~$\mathsf{G_{RS4}}=\mathsf{G_{RKT}}+(\odot rt)$\\
$\mathsf{G_{KW}}$~&~$\mathsf{G_{RKW}}=\mathsf{G_{WH}}+(\odot w)$\\
\hline
\end{tabular}
\]
\vspace{-1em}
\label{table:gsc}
\end{table}

For example, the algebraic correspondents in Table \ref{table:cp} are good $\mc{L}_\bu$-sequents. Then we get Gentzen-style sequent calculi in Table \ref{table:gsc} for residuated BDIs defined by the corresponding $\mc{L}_\bu$-sequents. These calculi admit $\mrm{(mix)}$ and have the subformula property.

\subsection{Comparison with literature}
Our framework in the present paper is to apply unified correspondence theory to proof theory of strict implication logics. The sequent calculi developed for conservative extensions are Gentzen-style.
This framework is quite different from the approaches in literature. Here we compare some sequent calculi for strict implication logics in literature with these calculi listed in Table \ref{table:gsc}.

Two types of calculi for non-classical logics in literature are distinguished by Alenda, Olivetti and Pozzato \cite{AOP12}:
\begin{quote}
 ``Similarly to modal logics and other extensions/alternative to classical logics two types of calculi: {\em external} calculi which make use of labels and relations on them to import the semantics into the syntax, and {\em internal} calculi which stay within the language, so that a configuration' (sequent, tableaux node ...) can be directly interpreted as a formula of the language." \cite[p.15]{AOP12} 
\end{quote} 
Obviously the sequent calculi developed in the present paper are {\em internal} because every structure in an $\mc{L}_\mrm{LC}$-sequent is directly translated into an $\mc{L}_\mrm{LC}$-formula. Ishigaki and Kashima \cite{IK07} also developed internal sequent calculi for some strict implication logics, but we have mentioned the advantages of our approach in Section \ref{section:introduction}.

External calculi for strict implication logics are also developed in literature. 
Labelled sequent calculi for intermediate logics are developed by Dyckhoff and Negri \cite{DN12}, and their connections with Hilbert axioms and hypersequents are investigated by Ciabattoni et al \cite{CMS2013}. In this approach, any intermediate logic characterized by a class of relational frames that is definable by first-order {\em geometric axioms}\footnote{A geometric axiom is a first-order formula of the form $\forall \ol{z}(P_1\wedge\ldots\wedge P_m\supset \exists \ol{x}(M_1\vee\ldots\vee M_n))$ where each $P_i$ is an atomic formula, and each $M_j$ is a conjunction of atomic formulas, and $\ol{z}$ and $\ol{x}$ are sequences of bounded variables. Each geometric axiom can be transformed into a geometric rule \cite{DN12}.}, can be formalized in a cut-free and contraction-free labelled sequent calculus that extends the labelled sequent calculus for intuitionistic logic with geometric rules transformed from these geometric axioms. Using the same approach, Yamasaki and Sano \cite{ys16} developed labelled sequent calculi for some subintuitionistic logics \cite{corsi87}.

The development of an external calculus for a strict implication logic depends on that the logic has geometric relational semantics, i.e., it is sound and complete with respect to a class of relational frames which is definable by a set of geometric theories. 
Our internal calculi for strict implication logics are developed 
for subvarieties of BDI algebras and they do not necessarily have relational semantics. The strict implication logic $\mbf{S}_\mathbb{BDI}$ is indeed an example without binary relational semantics.

The algorithm $\msf{ALBA}$, one of the main tools in unified correspondence theory, is applied in the present paper to the proof theory of strict implication logics. Firstly, it is used as a tool to calculate the first-order correspondents of inductive $\mc{L}_\mrm{SI}$-sequents. If the correspondents of a set of inductive sequents are geometric axioms, they can be transformed into geometric rules, and hence some labelled sequent calculi for strict implication logics can be developed. It is unknown if $\msf{ALBA}$ can capture all geometric axioms. A general converse correspondence theory is unknown yet. Secondly, our novel application of $\msf{ALBA}$ is to calculate the algebraic correspondents of some inductive $\mc{L}_\mrm{SI}$-sequents in the language $\mc{L}_\bu$. A proof-theoretic consequence of this application is that one can obtain mix-free internal sequent calculi for the conservative extensions of some  
strict implication logics.
However, the systematic connections between algebraic and first-order correspondents is unknown.

Our framework in the present paper may not be able to cover all such logics which have binary relational semantics. Consider strict implication logics containing (Sym) or (Euc) based on $\mbf{S}_\mathbb{WH}$ in Table \ref{table:gsc}.
Since (Sym) and (Euc) may not have correspondent in $\mc{L}_\bu$, these logics may not admit Gentzen-style sequent calculi that are obtained from $\msf{G_{WH}}$ by adding structural rules about $\odot$. Another example is Visser's logic $\msf{FPL}$ (formal provability logic) \cite{Visser81} which is a strict implication logic that extends basic propositional logic with the L\"{o}b's axiom $(q\imp p)\imp p\vdash q\imp p$. This axiom is not inductive. A labelled sequent calculus may be developed for $\msf{FPL}$ because it has binary relational semantics. But it is impossible to develop a Gentzen-style sequent calculus for it in our framework.

\section{Conclusion}
The present work studies the proof theory for strict impaction logic using unified correspondence theory as a proof-theoretic tool. First of all, we present general results about the semantic conservativity on DLE-logics via canonical extension. A consequence is that the strict implication logic $\msf{S_{BDI}}$ is conservatively extended to the Lambek calculus $\msf{BDFNL}$. The algorithm $\msf{ALBA}$ as a calculus for correspondence between DLE-logic and first-order logic and hence for canonicity, is specialized to the strict implication logic and Lambek calculus. The main contribution of the present paper is that we obtain an Ackermann lemma based calculus $\msf{ALC}$ from the algorithm $\msf{ALBA}$ as a  
tool for proving algebraic correspondence between a wide range of strict implication sequents and sequents in the language $\mc{L}_\bu$.
This tool gives not only more conservativity results, but also analytic rules needed for introducing the Gentzen-style cut-free sequent calculi. Another contribution is that we introduce a Gentzen-style sequent calculus for $\msf{BDFNL}$ and some of its extensions with analytic rules.

The final remark is about good $\mc{L}_\bu$-sequents that are used for obtaining cut-free sequent calculus.
It is very likely that a hiearchy of $\mc{L}_\bu$-sequents from which one obtains analytic rules can be established. Other connectives $\wedge, \vee$ and $\imp$ can be in principle added into the language $\mc{L}_\bu$ such that more analytic rules will be obtained. This is our work in progress.
Moreover, our approach to the proof theory of strict implication may be generalized to arbitrary DLE-logics.

\section*{Acknowledgement}
Thanks are given to Dr. Alessandra Palmigiano (Delft University of Technology) for her comments on drafts of this paper, and the reviewers for their helpful comments on the submitted version. The work of the first author is supported by Chinese national funding for social sciences and humanities (grant no. 14ZDB016). The work of the second author has been made possible by the NWO Vidi grant 016.138.314, by the NWO Aspasia grant 015.008.054, and by a Delft Technology Fellowship awarded in 2013.

\begin{appendix}
\section{Algebraic Correspondence}\label{appendix:cp}
(I) 
\[
\AxiomC{$\Imp q\vdash p\imp p$}
\RightLabel{\footnotesize RR1}
\UnaryInfC{$\Imp p\bu q\vdash p$}
\DisplayProof
\]
(MP)
\[
\AxiomC{$\Imp p\wedge (p\imp q)\vdash q$}
\RightLabel{\footnotesize Ap1}
\UnaryInfC{$r\vdash p\wedge (p\imp q)\Imp r\vdash q$}
\RightLabel{\footnotesize $\wedge$S}
\UnaryInfC{$r\vdash p, r\vdash p\imp q\Imp r\vdash q$}
\RightLabel{\footnotesize RL1}
\UnaryInfC{$r\vdash p, p\bu r\vdash q\Imp r\vdash q$}
\RightLabel{\footnotesize LAck}
\UnaryInfC{$p\bu p\vdash q\Imp p\vdash q$}
\RightLabel{\footnotesize RAck}
\UnaryInfC{$\Imp p\vdash p\bu p$}
\DisplayProof
\]
(W) 
\[
\AxiomC{$\Imp p\vdash q\imp p$}
\RightLabel{\footnotesize RR1}
\UnaryInfC{$\Imp q\bu p\vdash p$}
\DisplayProof
\]
(RT) 
\[
\AxiomC{$\Imp p\imp q\vdash r\imp (p\imp q)$}
\RightLabel{\footnotesize Ap1}
\UnaryInfC{$s\vdash p\imp q\Imp s\vdash r\imp (p\imp q)$}
\RightLabel{\footnotesize RL1}
\UnaryInfC{$p\bu s\vdash q\Imp s\vdash r\imp (p\imp q)$}
\RightLabel{\footnotesize RR1}
\UnaryInfC{$p\bu s\vdash q\Imp r\bu s\vdash p\imp q$}
\RightLabel{\footnotesize RR1}
\UnaryInfC{$p\bu s\vdash q\Imp p\bu(r\bu s)\vdash q$}
\RightLabel{\footnotesize RAck}
\UnaryInfC{$\Imp p\bu(r\bu s)\vdash p\bu s$}
\DisplayProof
\]
(B)
\[
\AxiomC{$\Imp p\imp q\vdash (r\imp p)\imp (r\imp q)$}
\RightLabel{\footnotesize Ap1}
\UnaryInfC{$s\vdash p\imp q\Imp s\vdash (r\imp p)\imp (r\imp q)$}
\RightLabel{\footnotesize RL1}
\UnaryInfC{$p\bu s\vdash q\Imp s\vdash (r\imp p)\imp (r\imp q)$}
\RightLabel{\footnotesize RR1}
\UnaryInfC{$p\bu s\vdash q\Imp (r\imp p)\bu s\vdash r\imp q$}
\RightLabel{\footnotesize RR1}
\UnaryInfC{$p\bu s\vdash q\Imp r\bu((r\imp p)\bu s)\vdash q$}
\RightLabel{\footnotesize RAck}
\UnaryInfC{$\Imp r\bu((r\imp p)\bu s)\vdash p\bu s$}
\RightLabel{\footnotesize RR1}
\UnaryInfC{$\Imp (r\imp p)\bu s\vdash r\imp(p\bu s)$}
\RightLabel{\footnotesize RR2}
\UnaryInfC{$\Imp r\imp p \vdash (r\imp(p\bu s))\limp s$}
\RightLabel{\footnotesize Ap1}
\UnaryInfC{$t\vdash r\imp p\Imp t \vdash (r\imp(p\bu s))\limp s$}
\RightLabel{\footnotesize RL1}
\UnaryInfC{$r\bu t\vdash p\Imp t \vdash (r\imp(p\bu s))\limp s$}
\RightLabel{\footnotesize RR2}
\UnaryInfC{$r\bu t\vdash p\Imp s\bu t \vdash r\imp(p\bu s) $}
\RightLabel{\footnotesize RR1}
\UnaryInfC{$r\bu t\vdash p\Imp r\bu (s\bu t) \vdash p\bu s$}
\RightLabel{\footnotesize RAck}
\UnaryInfC{$\Imp r\bu (s\bu t) \vdash (r\bu t)\bu s$}
\DisplayProof
\]
(B$'$)
\[
\AxiomC{$\Imp p\imp q\vdash (q\imp r)\imp (p\imp r)$}
\RightLabel{\footnotesize Ap1}
\UnaryInfC{$s\vdash p\imp q\Imp s\vdash (q\imp r)\imp (p\imp rt)$}
\RightLabel{\footnotesize RL1}
\UnaryInfC{$p\bu s\vdash q\Imp s\vdash (q\imp r)\imp (p\imp r)$}
\RightLabel{\footnotesize RR1}
\UnaryInfC{$p\bu s\vdash q\Imp (q\imp r)\bu s\vdash p\imp r$}
\RightLabel{\footnotesize RR1}
\UnaryInfC{$p\bu s\vdash q\Imp p\bu((q\imp r)\bu s)\vdash r$}
\RightLabel{\footnotesize RAck}
\UnaryInfC{$\Imp p\bu((p\bu s \imp r)\bu s)\vdash r$}
\RightLabel{\footnotesize RR1}
\UnaryInfC{$\Imp (p\bu s \imp r)\bu s \vdash p\imp r$}
\RightLabel{\footnotesize RR2}
\UnaryInfC{$\Imp p\bu s \imp r \vdash (p\imp r)\limp s$}
\RightLabel{\footnotesize Ap1}
\UnaryInfC{$t\vdash p\bu s \imp r\Imp t \vdash (p\imp r)\limp s$}
\RightLabel{\footnotesize RL1}
\UnaryInfC{$(p\bu s)\bu t\vdash r\Imp t \vdash (p\imp r)\limp s$}
\RightLabel{\footnotesize RR2}
\UnaryInfC{$(p\bu s)\bu t\vdash r\Imp t\bu s \vdash p\imp r$}
\RightLabel{\footnotesize RR1}
\UnaryInfC{$(p\bu s)\bu t\vdash r\Imp p\bu(t\bu s) \vdash  r$}
\RightLabel{\footnotesize RAck}
\UnaryInfC{$\Imp p\bu(t\bu s) \vdash (p\bu s)\bu t$}
\DisplayProof
\]
(C) 
\[
\AxiomC{$\Imp p\imp (q\imp r)\vdash q\imp (p\imp r)$}
\RightLabel{\footnotesize Ap1}
\UnaryInfC{$s\vdash p\imp (q\imp r) \Imp s\vdash q\imp (p\imp r)$}
\RightLabel{\footnotesize RL1}
\UnaryInfC{$p\bu s\vdash q\imp r \Imp s\vdash q\imp (p\imp r)$}
\RightLabel{\footnotesize RL1}
\UnaryInfC{$q\bu(p\bu s)\vdash r \Imp s\vdash q\imp (p\imp r)$}
\RightLabel{\footnotesize RR1}
\UnaryInfC{$q\bu(p\bu s)\vdash r \Imp q\bu s\vdash  p\imp r $}
\RightLabel{\footnotesize RR1}
\UnaryInfC{$q\bu(p\bu s)\vdash r \Imp p\bu(q\bu s)\vdash r $}
\RightLabel{\footnotesize RAck}
\UnaryInfC{$\Imp p\bu(q\bu s)\vdash q\bu(p\bu s)$}
\DisplayProof
\]
(Fr)
\[
\AxiomC{$\Imp p\imp (q\imp r)\vdash (p\imp q)\imp (p\imp r)$}
\RightLabel{\footnotesize (Ap1, Ap2)}
\UnaryInfC{$s\vdash p\imp (q\imp r), (p\imp q)\imp (p\imp r)\vdash t\Imp s\vdash t$}
\RightLabel{\footnotesize (RL1)}
\UnaryInfC{$q\bu (p\bu s)\vdash r, (p\imp q)\imp (p\imp r)\vdash t\Imp s\vdash t$}
\RightLabel{\footnotesize ($\imp$Ap1,$\imp$Ap2)}
\UnaryInfC{$q\bu (p\bu s)\vdash r, u\vdash p\imp q,p\imp r\vdash v, u\imp v\vdash t \Imp s\vdash t$}
\RightLabel{\footnotesize (RL1)}
\UnaryInfC{$q\bu (p\bu s)\vdash r, p\bu u\vdash q,p\imp r\vdash v, u\imp v\vdash t \Imp s\vdash t$}
\RightLabel{\footnotesize (AAP2)}
\UnaryInfC{$q\bu (p\bu s)\vdash r, p\bu u\vdash q,p\imp r\vdash v\Imp s\vdash u\imp v$}
\RightLabel{\footnotesize (RR1)}
\UnaryInfC{$q\bu (p\bu s)\vdash r, p\bu u\vdash q,p\imp r\vdash v\Imp u\bu s\vdash v$}
\RightLabel{\footnotesize (AAp2)}
\UnaryInfC{$q\bu (p\bu s)\vdash r, p\bu u\vdash q\Imp u\bu s\vdash p\imp r$}
\RightLabel{\footnotesize (RR1)}
\UnaryInfC{$q\bu (p\bu s)\vdash r, p\bu u\vdash q\Imp p\bu(u\bu s)\vdash r$}
\RightLabel{\footnotesize (AAp2)}
\UnaryInfC{$p\bu u\vdash q\Imp p\bu(u\bu s)\vdash q\bu (p\bu s)$}
\RightLabel{\footnotesize (RAck)}
\UnaryInfC{$\Imp p\bu(u\bu s)\vdash (p\bu u)\bu (p\bu s)$}
\DisplayProof
\]
(W$'$) 
\[
\AxiomC{$\Imp p\imp (p\imp q)\vdash p\imp q$}
\RightLabel{\footnotesize Ap1}
\UnaryInfC{$r\vdash p\imp (p\imp q)\Imp r\vdash p\imp q$}
\RightLabel{\footnotesize RL1}
\UnaryInfC{$p\bu r\vdash p\imp q \Imp r\vdash p\imp q$}
\RightLabel{\footnotesize RL1}
\UnaryInfC{$p\bu(p\bu r)\vdash q \Imp r\vdash p\imp q$}
\RightLabel{\footnotesize RR1}
\UnaryInfC{$p\bu(p\bu r)\vdash q \Imp p\bu r\vdash q$}
\RightLabel{\footnotesize RAck}
\UnaryInfC{$\Imp p\bu r\vdash p\bu(p\bu r)$}
\DisplayProof
\]
\end{appendix}

\end{document}